\documentclass[11pt,reqno]{amsart}
\usepackage[dvipdfm,a4paper,top=30mm,bottom=25mm,left=25mm,right=25mm]{geometry}
\usepackage{amsmath,amssymb,amsthm,mathtools}
\usepackage{comment,mathrsfs,color,cancel}

\usepackage[colorlinks=true,linkcolor=blue,citecolor=blue]{hyperref}

\newtheorem{theorem}{Theorem}[section]
\newtheorem{lemma}[theorem]{Lemma}
\newtheorem{proposition}[theorem]{Proposition}

\theoremstyle{definition}
\newtheorem{definition}[theorem]{Definition}

\theoremstyle{remark}
\newtheorem{remark}[theorem]{Remark}

\numberwithin{equation}{section}

\allowdisplaybreaks[4]

\def\rnum#1{\expandafter{\romannumeral #1}} 
\def\Rnum#1{\uppercase\expandafter{\romannumeral #1}}
\def\~#1{\widetilde #1}

\def\<{\langle}
\def\>{\rangle}

\begin{document}

\title[NLS equation with a long range potential]{Scattering solutions to nonlinear Schr\"odinger equation with a long range potential}


\author{Masaru Hamano}
\address{Department of Mathematics, Graduate School of Science and Engineering Saitama University, Shimo-Okubo 255, Sakura-ku, Saitama-shi, Saitama 338-8570, Japan}
\email{ess70116@mail.saitama-u.ac.jp}

\author{Masahiro Ikeda}
\address{Department of Mathematics, Faculty of Science and Technology, Keio University, 3-14-1 Hiyoshi, Kohoku-ku, Yokohama, 223-8522, Japan/Center for Advanced Intelligence Project, Riken, Japan}
\email{masahiro.ikeda@riken.jp / masahiro.ikeda@riken.jp}







\begin{abstract}
In this paper, we consider a nonlinear Schr\"odinger equation with a repulsive inverse-power potential.
It is known that the corresponding stationary problem has a ``radial'' ground state.
Here, the ``radial'' ground state is a least energy solution among radial solutions to the stationary problem.
We prove that if radial initial data below the ``radial'' ground state has positive virial functional, then the corresponding solution to the nonlinear Schr\"odinger equation scatters.
In particular, we can treat not only short range potentials but also long range potentials.
\end{abstract}

\maketitle

\tableofcontents


\section{Introduction}

In this paper, we consider the following nonlinear Schr\"odinger equation with an inverse power potential:
\begin{equation}
	i\partial_tu + \Delta_\gamma u
		= - |u|^{p-1}u, \quad (t,x) \in \mathbb{R} \times \mathbb{R}^d, \tag{NLS$_{\gamma}$} \label{NLS}
\end{equation}
where $d \geq 1$, $1 + \frac{4}{d} < p < \infty$ if $d = 1, 2$, $1 + \frac{4}{d} < p < 1 + \frac{4}{d-2}$ if $d \geq 3$, $\Delta_\gamma = \Delta - \frac{\gamma}{|x|^\mu}$, $\gamma > 0$, $0 < \mu < \min\{d,2\}$, $u = u(t,x)$ is a complex-valued unknown function, and $u_0(x) = u(0,x)$ is a complex-valued given function.
In particular, we deal with the Cauchy problem \eqref{NLS} with initial condition
\begin{equation}
	u(0,x)
		= u_0(x) \in H^1(\mathbb{R}^d). \tag{IC} \label{IC}
\end{equation}
In this paper, we are interested in time behavior of solutions to \eqref{NLS}.

Since $\frac{\gamma}{|x|^\mu} > 0$ and $\frac{\gamma}{|x|^\mu} \in L_\text{loc}^1(\mathbb{R}^d)$, $- \Delta_\gamma$ is defined as a unique self-adjoint operator associated with the non-negative quadratic form $\<(-\Delta_\gamma)f,f\>_{L^2}$ on $C_c^\infty(\mathbb{R}^d)$, where $C_c^\infty(\mathbb{R}^d)$ is a space of smooth functions with a compact support.
From Stone's theorem, the Schr\"odinger group $\{e^{it\Delta_\gamma}\}_{t \in \mathbb{R}}$ is generated by $- \Delta_\gamma$ on $L^2(\mathbb{R}^d)$.
Moreover, $- \Delta_\gamma$ is purely absolutely continuous: $\sigma(-\Delta_\gamma) = \sigma_\text{ac}(-\Delta_\gamma) = [0,\infty)$ and has no eigenvalue.

First, we recall that existence of a solution to \eqref{NLS} is assured.
It is known that the Cauchy problem \eqref{NLS} with \eqref{IC} is locally well-posed in $H^1(\mathbb{R}^d)$ (see \cite{Caz03}).
More precisely, for any $u_0 \in H^1(\mathbb{R}^d)$, there exist $T_\text{max} \in (0,\infty]$ and $T_\text{min} \in [-\infty,0)$ such that \eqref{NLS} has a unique solution $u$ in $C_t((T_\text{min},T_\text{max});H_x^1(\mathbb{R}^d)) \cap C_t^1((T_\text{min},T_\text{max});H_x^{-1}(\mathbb{R}^d))$.
Furthermore, the solution has a blow-up alternative: If $T_\text{max} < \infty$ (resp. $T_\text{min} > - \infty$), then
\begin{align*}
	\lim_{t \nearrow T_\text{max}}\|u(t)\|_{H_x^1}
		= \infty,\ \ \ 
	\left( \text{resp. } \lim_{t \searrow T_\text{min}}\|u(t)\|_{H_x^1}
		= \infty \right)
\end{align*}
and preserves its mass and energy, which are defined respectively as
\begin{align*}
	\text{(Mass) }&\ \ M[u(t)]
		:= \int_{\mathbb{R}^d}|u(t,x)|^2dx, \\
	\text{(Energy) }&\ \ E_{\gamma}[u(t)]
		:= \frac{1}{2}\int_{\mathbb{R}^d}|\nabla u(t,x)|^2 + \frac{\gamma}{|x|^\mu}|u(t,x)|^2dx - \frac{1}{p+1}\int_{\mathbb{R}^d}|u(t,x)|^{p+1}dx.
\end{align*}
We turn to time behavior of the solutions to \eqref{NLS}.
We define time behaviors of solutions to \eqref{NLS}. 

\begin{definition}[Scattering, Blow-up, Grow-up, and Standing wave]
Let $(T_\text{min},T_\text{max})$ denote the maximal lifespan of $u$.
\begin{itemize}
\item (Scattering)
We say that the solution $u$ to \eqref{NLS} scatters in positive time (resp. negative time) if $T_\text{max} = +\infty$ (resp. $T_\text{min} = -\infty$) and there exists $\psi_+ \in H^1$ (resp. $\psi_-\in H^1$) such that
\begin{align*}
	\lim_{t \rightarrow +\infty}\|u(t) - e^{it\Delta_{\gamma}}\psi_+\|_{H_x^1}
		= 0\ \ \ 
	\left(\text{resp.}\ \lim_{t \rightarrow -\infty}\|u(t) - e^{it\Delta_{\gamma}}\psi_-\|_{H_x^1} = 0\right).
\end{align*}
\item (Blow-up)
We say that the solution $u$ to \eqref{NLS} blows up in positive time (resp. negative time) if $T_{\text{max}}<\infty$ (resp. $T_\text{min}>-\infty$).
Then, we note that we have
\begin{align*}
	\lim_{t \rightarrow +\infty}\|u(t)\|_{H_x^1}
		= \infty,\ \ \ 
	\left(\text{resp. }\lim_{t \rightarrow -\infty}\|u(t)\|_{H_x^1}
		= \infty\right)
\end{align*}
from the blow-up alternative.
\item (Grow-up)
We say that the solution $u$ to \eqref{NLS} grows up in positive time (resp. negative time) if $T_\text{max} = +\infty$ (resp. $T_\text{min} = -\infty$) and
\begin{align*}
	\limsup_{t \rightarrow +\infty}\|u(t)\|_{H_x^1}
		= \infty,\ \ \ 
	\left(\text{resp. }\limsup_{t \rightarrow -\infty}\|u(t)\|_{H_x^1}
		= \infty\right).
\end{align*}
\item (Standing wave)
We say that the solution $u$ to \eqref{NLS} is a standing wave if $u(t,x) = e^{i\omega t}Q_{\omega,\gamma}(x)$ for $\omega \in \mathbb{R}$, where $Q_{\omega,\gamma}$ satisfies
\begin{align}
	- \omega Q_{\omega,\gamma} + \Delta_\gamma Q_{\omega,\gamma}
		= - |Q_{\omega,\gamma}|^{p-1}Q_{\omega,\gamma} \tag{SP$_{\omega,\gamma}$} \label{SP}.
\end{align}
\end{itemize}
\end{definition}

We recall that the authors \cite{HamIkeProceeding, HamIkeISAAC} proved that \eqref{NLS} has a radial standing wave $u(t,x) = e^{i\omega t}Q_{\omega,\gamma}(x)$.
In particular, \eqref{SP} has a ``radial'' ground state $Q_{\omega,\gamma}$.
A set $\mathcal{G}_{\omega,\gamma,\text{rad}}$ of whole of the ``radial'' ground state is defined as
\begin{align*}
	\mathcal{G}_{\omega,\gamma,\text{rad}}
		:= \{\phi \in \mathcal{A}_{\omega,\gamma,\text{rad}} : S_{\omega,\gamma}(\phi) \leq S_{\omega,\gamma}(\psi)\text{ for any }\psi \in \mathcal{A}_{\omega,\gamma,\text{rad}}\},
\end{align*}
where
\begin{align*}
	\mathcal{A}_{\omega,\gamma,\text{rad}}
		& := \{\psi \in H_\text{rad}^1(\mathbb{R}^d) \setminus \{0\} : S_{\omega,\gamma}'(\psi) = 0\}, \\
	S_{\omega,\gamma}(f)
		& := \frac{\omega}{2}M[f] + E_{\gamma}[f].
\end{align*}
To get the ``radial'' ground state to \eqref{SP}, we considered the minimization problem
\begin{align*}
	r_{\omega,\gamma}^{\alpha,\beta}
		:= \inf\{S_{\omega,\gamma}(f) : f \in H_\text{rad}^1(\mathbb{R}^d)\setminus \{0\},\ K_{\omega,\gamma}^{\alpha,\beta}(f) = 0\},
\end{align*}
where a functional $K_{\omega,\gamma}^{\alpha,\beta}$ is defined as
\begin{align*}
	K_{\omega,\gamma}^{\alpha,\beta}(f)
		:= \left.\frac{\partial}{\partial\lambda}\right|_{\lambda = 0}S_{\omega,\gamma}(e^{\alpha\lambda}f(e^{\beta\lambda}\,\cdot\,))
\end{align*}
for $(\alpha,\beta)$ satisfying
\begin{align}
	\alpha
		> 0,\ \ 
	\beta
		\geq 0,\ \ 
	2\alpha - d\beta
		\geq 0. \label{104}
\end{align}
$r_{\omega,\gamma}^{\alpha,\beta}$ has a minimizer and it is the ``radial'' ground state to \eqref{SP}.

\begin{theorem}[Existence of a ``radial'' ground state to \eqref{SP}, \cite{HamIkeProceeding, HamIkeISAAC}]
Let $d \geq 2$, $1 < p < \infty$ if $d = 2$, $1 < p < 1 + \frac{4}{d-2}$ if $d \geq 3$, $\gamma > 0$, $0 < \mu < 2$, and $\omega > 0$.
Then, \eqref{SP} has a ``radial'' ground state $Q_{\omega,\gamma}$.
Furthermore, the ``radial'' ground state $Q_{\omega,\gamma}$ is characterized by a functional $K_{\omega,\gamma}^{\alpha,\beta}$ for $(\alpha,\beta)$ satisfying \eqref{104}.
More precisely, $\mathcal{G}_{\omega,\gamma,\text{rad}} = \mathcal{M}_{\omega,\gamma,\text{rad}}^{\alpha,\beta}$ holds, where
\begin{align*}
	\mathcal{M}_{\omega,\gamma,\text{rad}}^{\alpha,\beta}
		&:= \{\phi \in H_\text{rad}^1(\mathbb{R}^d) \setminus\{0\} : S_{\omega,\gamma}(\phi) = r_{\omega,\gamma}^{\alpha,\beta},\ K_{\omega,\gamma}^{\alpha,\beta}(\phi) = 0\}.
\end{align*}
\end{theorem}

Since $r_{\omega,\gamma}^{\alpha,\beta}$ is independent of $(\alpha,\beta)$, we express $r_{\omega,\gamma} := r_{\omega,\gamma}^{\alpha,\beta}$.
In this paper, we espacially use $(\alpha,\beta) = (1,0)$ and $(\alpha,\beta) = (d,2)$.
$K_{\omega,\gamma}^{1,0}$ is called the Nehari functional and $K_{\omega,\gamma}^{d,2}$ is called the virial functional.
For simplicity, we use notations $K_{\omega,\gamma}^{1,0} =: \mathcal{N}_{\omega,\gamma}$ and $K_{\omega,\gamma}^{d,2} =: K_\gamma$.
We note that $\mathcal{N}_{\omega,\gamma}$ and $K_\gamma$ are written as
\begin{align*}
	\mathcal{N}_{\omega,\gamma}(f)
		& := \omega\|f\|_{L^2}^2 + \|(-\Delta_\gamma)^\frac{1}{2}f\|_{L^2}^2 - \|f\|_{L^{p+1}}^{p+1}, \\
	K_\gamma(f)
		& := 2\|\nabla f\|_{L^2}^2 + \mu\int_{\mathbb{R}^d}\frac{\gamma}{|x|^\mu}|f(x)|^2dx - \frac{d(p-1)}{p+1}\|f\|_{L^{p+1}}^{p+1}.
\end{align*}

For $\gamma > 0$, the following minimization problem does not have a minimizer and is independent of $(\alpha,\beta)$ (see \cite[Theorem 1.5]{HamIkeISAAC}):
\begin{align*}
	n_{\omega,\gamma}^{\alpha,\beta}
		:= \inf\{S_{\omega,\gamma}(f) : f \in H^1(\mathbb{R}^d) \setminus \{0\},\ K_{\omega,\gamma}^{\alpha,\beta}(f) = 0\}.
\end{align*}
So, we do not know that \eqref{SP} has the ground state in the usual sense.
A set $\mathcal{G}_{\omega,\gamma}$ of whole of the ground state is defined as
\begin{align*}
	\mathcal{G}_{\omega,\gamma}
		:= \{\phi \in \mathcal{A}_{\omega,\gamma} : S_{\omega,\gamma}(\phi) \leq S_{\omega,\gamma}(\psi)\text{ for any }\psi \in \mathcal{A}_{\omega,\gamma}\},
\end{align*}
where
\begin{align*}
	\mathcal{A}_{\omega,\gamma}
		& := \{\psi \in H^1(\mathbb{R}^d) \setminus \{0\} : S_{\omega,\gamma}'(\psi) = 0\}.
\end{align*}
It is known that if $\gamma = 0$, then $n_{\omega,0}^{\alpha,\beta}$ is independent of $(\alpha,\beta)$ and $n_{\omega,0}^{\alpha,\beta}$ is attained by the ground state $Q_{\omega,0}$ to \eqref{SP} with $\gamma = 0$, that is, $n_{\omega,0}^{\alpha,\beta} = S_{\omega,0}(Q_{\omega,0})$.
For simplicity, we express $n_{\omega,\gamma}^{\alpha,\beta} = n_{\omega,\gamma}$ for $\gamma \geq 0$ from now on.
By using the ground state $Q_{1,0}$ to \eqref{SP} with $\omega = 1$ and $\gamma = 0$, Dinh \cite{Din21} showed a uniformly bounded result and a blow-up result (see also \cite{MiaZhaZhe18}).

\begin{theorem}[Boundedness versus unboundedness \Rnum{1}, \cite{Din21}]\label{Din21}
Let $d \geq 1$, $1 + \frac{4}{d} < p < \infty$ if $d = 1, 2$, $1 + \frac{4}{d} < p < 1 + \frac{4}{d-2}$ if $d \geq 3$, $\gamma > 0$, $0 < \mu < \min\{2,d\}$, and $u_0 \in H^1(\mathbb{R}^d)$.
Let $Q_{1,0}$ be the ground state to \eqref{SP} with $\omega = 1$ and $\gamma = 0$.
We assume that $u_0$ satisfies
\begin{align}
	M[u_0]^{1-s_c}E_\gamma[u_0]^{s_c}
		< M[Q_{1,0}]^{1-s_c}E_0[Q_{1,0}]^{s_c}, \label{131}
\end{align}
where $s_c := \frac{d}{2} - \frac{2}{p-1}$.
\begin{itemize}
\item (Boundedness)
If $u_0$ satisfies
\begin{align}
	\|u_0\|_{L^2}^{1-s_c}\|\nabla u_0\|_{L^2}^{s_c}
		< \|Q_{1,0}\|_{L^2}^{1-s_c}\|\nabla Q_{1,0}\|_{L^2}^{s_c} \label{130}
\end{align}
then, a solution $u$ to \eqref{NLS} with \eqref{IC} exists globally in time and satisfies
\begin{align*}
	\|u(t)\|_{L_x^2}^{1-s_c}\|\nabla u(t)\|_{L_x^2}^{s_c}
		< \|Q_{1,0}\|_{L^2}^{1-s_c}\|\nabla Q_{1,0}\|_{L^2}^{s_c}
\end{align*}
for any $t \in \mathbb{R}$.
\item (Unboundedness)
We assume that $u_0$ satisfies
\begin{align}
	\text{$|x|u_0\in L^2(\mathbb{R}^d)$ or ``$u_0 \in H_\text{rad}^1(\mathbb{R}^d)$ with $d \geq 2$ and $1 < p \leq 5$''}. \label{133}
\end{align}
If $u_0$ satisfies $E_\gamma(u_0) < 0$ or
\begin{align}
	\|u_0\|_{L^2}^{1-s_c}\|\nabla u_0\|_{L^2}^{s_c}
		> \|Q_{1,0}\|_{L^2}^{1-s_c}\|\nabla Q_{1,0}\|_{L^2}^{s_c} \label{132}
\end{align}
with $E_\gamma (u_0)\geq0$, then $u$ blows up.
\end{itemize}
\end{theorem}

By utilizing the ground state $Q_{1,0}$ to \eqref{SP} with $\omega = 1$ and $\gamma = 0$ and the ``radial'' ground state $Q_{\omega,\gamma}$ to \eqref{SP}, the authors \cite{HamIkeNonlinear} showed a uniformly bounded result, a blow-up or grow-up result, and a blow-up result.

\begin{theorem}[Boundedness versus unboundedness \Rnum{2}, \cite{HamIkeNonlinear}]\label{HamIkeNonlinear}
Let $d \geq 1$, $1 + \frac{4}{d} < p < \infty$ if $d = 1, 2$, $1 + \frac{4}{d} < p < 1 + \frac{4}{d-2}$ if $d \geq 3$, $\gamma > 0$, $0 < \mu < \min\{2,d\}$, and $u_0 \in H^1(\mathbb{R}^d)$.
Let $Q_{1,0}$ be the ground state to \eqref{SP} with $\omega = 1$ and $\gamma = 0$ and $Q_{\omega,\gamma}$ be the ``radial'' ground state to \eqref{SP}.
\begin{itemize}
\item (Boundedness)
Under \eqref{131}, the three conditions \eqref{130},
\begin{align*}
	\|u_0\|_{L^2}^{1-s_c}\|(-\Delta_\gamma)^\frac{1}{2}u_0\|_{L^2}^{s_c}
		< \|Q_{1,0}\|_{L^2}^{1-s_c}\|\nabla Q_{1,0}\|_{L^2}^{s_c},\ \ \text{ and }\ \ 
	K_\gamma(u_0)
		\geq 0
\end{align*}
are equivalent.
Moreover, if $d \geq 2$ and $u_0 \in H_\text{rad}^1(\mathbb{R}^d)$ satisfies
\begin{align}
	S_{\omega,\gamma}(u_0)
		< S_{\omega,\gamma}(Q_{\omega,\gamma})\ \text{ for some }\ \omega > 0 \label{134}
\end{align}
and $K_\gamma(u_0) \geq 0$, then a solution $u$ to \eqref{NLS} exists globally in time and is uniformly bounded in $H^1(\mathbb{R}^d)$ for time $t$.
\item (Unboundedness)
Under \eqref{131}, the three conditions \eqref{132},
\begin{align*}
	\|u_0\|_{L^2}^{1-s_c}\|(-\Delta_\gamma)^\frac{1}{2}u_0\|_{L^2}^{s_c}
		> \|Q_{1,0}\|_{L^2}^{1-s_c}\|\nabla Q_{1,0}\|_{L^2}^{s_c},\ \ \text{ and }\ \ 
	K_\gamma(u_0)
		< 0
\end{align*}
are equivalent.
This equivalence deduces that if $u_0$ satisfies \eqref{131}, \eqref{133}, and \eqref{132}, then a solution $u$ to \eqref{NLS} with \eqref{IC} blows up.
Moreover, it follows without \eqref{133} that $u$ blows up or grows up.
Furthermore, if $d \geq 2$, $1 < p \leq 5$, and $u_0 \in H_\text{rad}^1(\mathbb{R}^d)$ satisfies \eqref{134} and $K_\gamma(u_0) < 0$, then $u$ blows up.
\end{itemize}
\end{theorem}

\begin{remark}
For the conditions \eqref{131} and \eqref{134} in Theorem \ref{HamIkeNonlinear}, we remark the following.
The condition \eqref{134} is weaker than the condition \eqref{131}, that is, if $u_0 \in H_\text{rad}^1(\mathbb{R}^d)$ satisfies \eqref{131}, then \eqref{134} holds.
We can see the fact from the following two results:
\begin{itemize}
\item
$S_{\omega,\gamma}(u_0) < n_{\omega,\gamma}$ and \eqref{131} are equivalent,
\item
$n_{\omega,\gamma}$ is not attained and $r_{\omega,\gamma}$ is attained
\end{itemize}
(see also \cite[Proposition 1.8 and Remark 1.16]{HamIkeNonlinear}).
\end{remark}

Guo--Wang--Yao \cite{GuoWanYao18} investigated asymptotic behavior of the global solutions by restricting the range of $\mu$ to $1 < \mu < 2$.
The proof is based on Kenig--Merle argument \cite{KenMer06} (see also \cite{AkaNaw13, DuyHolRou08, FanXieCaz11, HolRou08}).

\begin{theorem}[Scattering versus blow-up or grow-up, \cite{GuoWanYao18}]
Let $d = 3$, $p = 3$, $\gamma > 0$, $1 < \mu < 2$, and $u_0 \in H^1(\mathbb{R}^d)$.
We assume that $u_0$ satisfies \eqref{131}.
\begin{itemize}
\item (Scattering)
If $u_0$ satisfies $K_\gamma(u_0) \geq 0$, then the solution $u$ to \eqref{NLS} with \eqref{IC} scatters.
\item (Blow-up or grow-up)
If $u_0$ satisfies $K_\gamma(u_0) < 0$, then the solution $u$ to \eqref{NLS} with \eqref{IC} blows up or grows up.
\end{itemize}
\end{theorem}

We note that in the case of $\mu = 2$, the following results are known.
Killip--Murphy--Visan--Zheng \cite{KilMurVisZhe17} proved a similar scattering result and a similar blow-up result.
Lu--Miao--Murphy \cite{LuMiaMur18} and Zheng \cite{Zhe18} proved a similar scattering result.
Dinh \cite{Din18} proved a global well-posedness result and a blow-up result.
Dinh \cite{Din181} proved strong instability of standing waves to \eqref{SP}.

We consider the following linear equation corresponding to \eqref{NLS}:
\begin{align}
\label{143}
\begin{cases}
&\hspace{-0.4cm}\displaystyle{
	i\partial_tu + \Delta_\gamma u
		= 0,
	} \\
&\hspace{-0.4cm}\displaystyle{
	u(0,x)
		= \psi(x)
	}.
\end{cases}
\end{align}
The time behavior of linear solutions to \eqref{143} changes as the borderline $\mu = 1$.
When $1 < \mu \leq \min\{d,2\}$, the potential $\frac{\gamma}{|x|^\mu}$ is called short range potential and satisfies that there exists $\psi_+ \in H^1(\mathbb{R}^d)$ (resp. $\psi_- \in H^1(\mathbb{R}^d)$) such that
\begin{align*}
	\lim_{t \rightarrow +\infty}\|e^{it\Delta_\gamma}\psi - e^{it\Delta}\psi_+\|_{H_x^1}
		= 0,\ \ \ 
	\left( \text{resp. }\lim_{t \rightarrow -\infty}\|e^{it\Delta_\gamma}\psi - e^{it\Delta}\psi_-\|_{H_x^1}
		= 0 \right)
\end{align*}
(see \cite{Miz201}).
That is, the nonlinear solution $u$ approaches not only a linear solution $e^{it\Delta_\gamma}u_\pm$ but also a free solution $e^{it\Delta}\psi_\pm$ as $t \rightarrow \pm \infty$ in \cite{GuoWanYao18, KilMurVisZhe17, LuMiaMur18, Zhe18}.
On the other hand, when $0 < \mu \leq 1$, the potential $\frac{\gamma}{|x|^\mu}$ is called long range potential and
\begin{align*}
	s-\lim_{t \rightarrow \pm\infty}e^{-it \Delta_\gamma}e^{it\Delta}\ \text{ in }\ L_x^2(\mathbb{R}^d)
\end{align*}
do not exist (see \cite{ReeSim78}).
That is, even if a nonlinear solution $u$ approaches a linear solution $e^{it\Delta_\gamma}u_\pm$, it does not approach a free solution $e^{it\Delta}\psi_\pm$.
In the case of $\mu = 1$, Miao--Zhang--Zheng \cite{MiaZhaZhe18} showed a scattering result for the nonlinear Schr\"odinger equation with a defocusing nonlinearity
\begin{align}
	i \partial_tu + \Delta u - \frac{\gamma}{|x|}u
		= + |u|^{p-1}u. \label{141}
\end{align}

\begin{theorem}[Miao--Zhang--Zheng, \cite{MiaZhaZhe18}]
Let $d= 3$, $\frac{7}{3} < p < 5$, $\gamma > 0$, and $u_0 \in H^1(\mathbb{R}^3)$.
Then, the solution $u$ to \eqref{141} with \eqref{IC} scatters.
\end{theorem}

In this paper, we prove a scattering result for \eqref{NLS} with a long range potential and a focusing nonlinearity.
For simplicity, we assume $d = p = 3$.
Here, we state our main result.

\begin{theorem}[Scattering]\label{Scattering}
Let $d = p = 3$, $\gamma > 0$, and $0 < \mu < 2$.
Let $Q_{\omega,\gamma}$ be the ``radial'' ground state to \eqref{SP}.
If $u_0 \in H_\text{rad}^1(\mathbb{R}^d)$ satisfies \eqref{134} and $K_{\gamma}(u_0)	\geq 0$, then the solution $u$ to \eqref{NLS} with \eqref{IC} scatters.
\end{theorem}

\begin{remark}
Combining Theorems \ref{Scattering} and \ref{HamIkeNonlinear}, we determine completely time behavior of the radial solutions to \eqref{NLS} with $d = p = 3$, $\gamma > 0$, and $0 < \mu < 2$ below the ``radial'' ground state by the sign of the virial functional $K_\gamma$ at initial data.
\end{remark}

We explain the difficulty and the idea of the proof.
We prove the Linear profile decomposition without using space transformation (see Theorem \ref{Linear profile decomposition}).
Since the potential $\frac{\gamma}{|x|^\mu}$ with $0 < \mu \leq 1$ has a slow decay, a difficulty arises in the argument for convergence of space transformation parameter.
To avoid using the space transformation, we utilize a radial assumption.
As the result, we can extend the range of $\mu$ to $0 < \mu < 2$ compared with \cite{GuoWanYao18}.\\

The organization of the rest of this paper is as follows:
In Section \ref{Sec:Preliminaries}, we collect notations and tools used throughout this paper.
In Section \ref{Sec:Parameter independence of the splitting below the ``radial'' ground state}, we separate a set $\{f \in H^1(\mathbb{R}^3) : S_{\omega,\gamma}(f) < S_{\omega,\gamma}(Q_{\omega,\gamma})\}$ below the ``radial'' ground state by the sign of $K_{\omega,\gamma}^{\alpha,\beta}(f)$.
We prove that these sets are independent of $(\alpha,\beta)$ (Proposition \ref{Initial data set}).
In Subsection \ref{Subsec:Local well-posedness}, we get local well-posedness to \eqref{NLS}.
In Subsection \ref{Subsec:Small data theory}, we prove that if initial data is sufficiently small, then the corresponding solution to \eqref{NLS} scatters.
In Subsection \ref{Subsec:Stability}, we prove stability result for scattering.
In Subsection \ref{Subsec:Final state problem}, we consider a final state problem and prove existence of wave operators.
In Section \ref{Sec:Linear profile decomposition}, we prove linear profile decomposition (Theorem \ref{Linear profile decomposition}).
This is a key tool for the proof of our main theorem.
In Section \ref{Sec:Scattering}, we assume for contradiction that a scattering threshold is less than $S_{\omega,\gamma}(Q_{\omega,\gamma})$ and construct a non-scattering solution on the threshold (critical solution).
From a rigidity theorem (Theorem \ref{Rigidity}), the critical solution does not exist and complete the proof of our main theorem.

\section{Preliminaries}\label{Sec:Preliminaries}

In this section, we define some notations and collect some known tools. 

\subsection{Notation and definition}

For nonnegative $X$ and $Y$, we write $X \lesssim Y$ to denote $X \leq CY$ for some $C > 0$.
If $X \lesssim Y \lesssim X$ holds, we write $X \sim Y$.
The dependence of implicit constants on parameters will be indicated by subscripts, e.g. $X \lesssim_u Y$ denotes $X \leq CY$ for some $C = C(u)$.
We write $a' \in [1,\infty]$ to denote the H\"older dual exponent to $a \in [1,\infty]$, that is, the solution $\frac{1}{a} + \frac{1}{a'} = 1$.

For $1 \leq p \leq \infty$, $L^p(\mathbb{R}^d)$ denotes the usual Lebesgue space.
For a Banach space $X$, we use $L^q(I;X)$ to denote the Banach space of functions $f : I \times \mathbb{R}^d \longrightarrow \mathbb{C}$ whose norm is $\|f\|_{L^q(I;X)} := \|\|f(t)\|_X\|_{L^q(I)} < \infty$.
We extend our notation as follows: If a time interval is not specified, then the $t$-norm is evaluated over $(-\infty,\infty)$.

We define respectively the Fourier transform of $f$ on $\mathbb{R}^d$ and the inverse Fourier transform of $f$ on $\mathbb{R}^d$ by
\begin{align*}
	\mathcal{F}f(\xi)
		= \widehat{f}(\xi)
		:= \int_{\mathbb{R}^d}e^{-2\pi ix\cdot\xi}f(x)dx\ \ \text{ and }\ \ 
	\mathcal{F}^{-1}f(x)
		= \check{f}(x)
		:= \int_{\mathbb{R}^d}e^{2\pi ix\cdot\xi}f(\xi)d\xi,
\end{align*}
where $x \cdot \xi := x_1\xi_1 + \cdots + x_d\xi_d$ denotes the usual inner product of $x$ and $\xi$ on $\mathbb{R}^d$.

$W^{s,p}(\mathbb{R}^d) := (1-\Delta)^{-\frac{s}{2}}L^p(\mathbb{R}^d)$ and $\dot{W}^{s,p}(\mathbb{R}^d) := (-\Delta)^{-\frac{s}{2}}L^p(\mathbb{R}^d)$ are the inhomogeneous Sobolev space and the homogeneous Sobolev space, respectively for $s \in \mathbb{R}$ and $p \in [1,\infty]$, where $(1-\Delta)^\frac{s}{2} = \<\nabla\>^s = \mathcal{F}^{-1}(1 + 4\pi^2|\xi|^2)^\frac{s}{2}\mathcal{F}$ and $(-\Delta)^\frac{s}{2} = |\nabla|^s = \mathcal{F}^{-1}(2\pi|\xi|)^s\mathcal{F}$, respectively.
When $p = 2$, we express $W^{s,2}(\mathbb{R}^d) = H^s(\mathbb{R}^d)$ and $\dot{W}^{s,2}(\mathbb{R}^d) = \dot{H}^s(\mathbb{R}^d)$.
We also define Sobolev spaces with a potential by $W_\gamma^{s,p}(\mathbb{R}^d) = (1-\Delta_\gamma)^{-\frac{s}{2}}L^p(\mathbb{R}^d)$ and $\dot{W}_\gamma^{s,p}(\mathbb{R}^d) = (-\Delta_\gamma)^{-\frac{s}{2}}L^p(\mathbb{R}^d)$ for $s \in \mathbb{R}$ and $p \in [1,\infty]$.
That is, $W_\gamma^{s,p}(\mathbb{R}^d)$ and $\dot{W}_\gamma^{s,p}(\mathbb{R}^d)$ are Hilbert spaces with a norm
\begin{align*}
	\|f\|_{W_\gamma^{s,p}}
		:= \|(1-\Delta_\gamma)^\frac{s}{2}f\|_{L^p}\ \text{ and }\ 
	\|f\|_{\dot{W}_\gamma^{s,p}}
		:= \|(-\Delta_\gamma)^\frac{s}{2}f\|_{L^p}
\end{align*}
respectively.
When $p = 2$, we express $W_\gamma^{s,2}(\mathbb{R}^d) = H_\gamma^s(\mathbb{R}^d)$ and $\dot{W}_\gamma^{s,2}(\mathbb{R}^d) = \dot{H}_\gamma^s(\mathbb{R}^d)$.

\subsection{Some tools}

From the following Generalized Hardy's inequality, the energy $E_\gamma$ is well-defined.

\begin{lemma}[Generalized Hardy's inequality, \cite{ZhaZhe14}]\label{Generalized Hardy inequality}
Let $d = 3$, $1 < q < \infty$, and $0 < \mu < 3$.
Then, the following inequality holds:
\begin{align*}
	\int_{\mathbb{R}^3}\frac{1}{|x|^\mu}|f(x)|^qdx
		\lesssim_{q,\mu}\||\nabla|^\frac{\mu}{q}f\|_{L^q}^q.
\end{align*}
\end{lemma}

From Lemma \ref{Generalized Hardy inequality} and the interpolation, the following lemma holds.

\begin{lemma}[Norm equivalence, \cite{Din21}]\label{Norm equivalence}
Let $d = 3$, $\gamma > 0$, and $0 < \mu < 2$.
Then, the following equivalent holds:
\begin{align*}
	\|(-\Delta_\gamma)^\frac{s}{2}f\|_{L^r}
		\sim \|(-\Delta)^\frac{s}{2}f\|_{L^r},\ \ \ 
	\|(1-\Delta_\gamma)^\frac{s}{2}f\|_{L^r}
		\sim \|(1-\Delta)^\frac{s}{2}f\|_{L^r},
\end{align*}
where $0 \leq s \leq 2$ and $1 < r < \frac{6}{s\mu}$.
\end{lemma}

Combining the usual Sobolev embedding and Lemma \ref{Norm equivalence}, we have the following lemma.

\begin{lemma}[Sobolev inequality]\label{Sobolev inequality}
Let $d = 3$, $\gamma > 0$, and $0 < \mu < 2$.
Then, we have
\begin{align*}
	\|f\|_{L^q}
		\lesssim \|(-\Delta_\gamma)^\frac{s}{2}f\|_{L^p},
\end{align*}
where $1 < p < q < \infty$, $1 < p < \frac{3}{s}$, $0 \leq s \leq 2$ and $\frac{1}{q} = \frac{1}{p} - \frac{s}{3}$.
\begin{align*}
	\|f\|_{L^q}
		\lesssim \|(1-\Delta_\gamma)^\frac{s}{2}f\|_{L^p}
\end{align*}
holds, where $1 < p < q < \infty$, $1 < p < \frac{3}{s}$, $0 \leq s \leq 2$ and $\frac{1}{q} \geq \frac{1}{p} - \frac{s}{3}$.
\end{lemma}

To state the Strichartz estimate (Theorem \ref{Strichartz estimate}), we define the following $\dot{H}^s$-admissible pairs.

\begin{definition}[$\dot{H}^s$-admissible]
Let $d = 3$ and $0 \leq s \leq 1$.
We say that a pair of exponents $(q,r)$ is called $\dot{H}^s$-admissible if $2 \leq q, r \leq \infty$ and
\begin{align*}
	\frac{2}{q} + \frac{3}{r}
		= \frac{3}{2} - s.
\end{align*}
We define a set $\Lambda_s := I_s \cap \{(q,r) \in \mathbb{R}^2 : (q,r)\text{ is }\dot{H}^s\text{-admissible}\}$, where
\begin{equation*}
I_s :=
\begin{cases}
	&\hspace{-0.4cm}\displaystyle{
		\left\{(q,r)\,;\ 2 \leq q \leq \infty,\ \frac{6}{3-2s} \leq r \leq\frac{6}{1-2s}\ \,\right\}\ \left(0 \leq s < \frac{1}{2}\right),
	} \\[0.4cm]
	&\hspace{-0.4cm}\displaystyle{
		\left\{(q,r)\,;\ \frac{4}{3-2s} < q \leq \infty,\ \frac{6}{3-2s} \leq r < \infty\right\}\ \left(\frac{1}{2} \leq s \leq 1\right).
	}
\end{cases}
\end{equation*}
\end{definition}

\begin{theorem}[Strichartz estimate, \cite{Miz20}]\label{Strichartz estimate}
Let $d = 3$, $0 < \mu < 2$, and $\gamma > 0$.
\begin{itemize}
\item (Homogeneous estimates)
If $(q,r) \in \Lambda_s$, then
\begin{align*}
	\|e^{it\Delta_\gamma}f\|_{L_t^qL_x^r}
		\lesssim \|f\|_{\dot{H}^s}.
\end{align*}
\item (Inhomogeneous estimates)
Let $t_0 \in \mathbb{R}$ and $I$ be a time interval including $t_0$.
If $(q_1,r_1) \in \Lambda_s$ and $(q_2,r_2) \in \Lambda_0$, then
\begin{align*}
	\left\|\int_{t_0}^te^{i(t-s)\Delta_\gamma}F(\cdot,s)ds\right\|_{L_t^{q_1}(I;L_x^{r_1})}
		\lesssim \||\nabla|^sF\|_{L_t^{q_2'}(I;L_x^{r_2'})},
\end{align*}
where implicit constants are independent of $f$ and $F$.
\end{itemize}
\end{theorem}

\begin{lemma}[Fractional calculus, \cite{ChrWei91}]\label{Fractional calculus}
Suppose $G \in C^1(\mathbb{C})$ and $s \in (0,1]$. Let $1 < r, r_2 < \infty$ and $1 < r_1 \leq \infty$ satisfying $\frac{1}{r} = \frac{1}{r_1} + \frac{1}{r_2}$.
Then, we have
\begin{align*}
	\||\nabla|^sG(u)\|_{L^r}
		\lesssim \|G'(u)\|_{L^{r_1}}\||\nabla|^su\|_{L^{r_2}}.
\end{align*}
\end{lemma}

\begin{lemma}[Radial Sobolev inequality, \cite{OgaTsu91}]\label{Radial Sobolev inequality}
Let $d = 3$ and $1 \leq p$.
For a radial function $f \in H^1(\mathbb{R}^3)$, it follows that
\begin{align*}
	\|f\|_{L^{p+1}(R \leq |x|)}^{p+1}
		\lesssim \frac{1}{R^\frac{2(p-1)}{2}}\|f\|_{L^2(R\leq|x|)}^\frac{p+3}{2}\|\nabla f\|_{L^2(R \leq |x|)}^\frac{p-1}{2}
\end{align*}
for any $R>0$, where the implicit constant is independent of $R$ and $f$.
\end{lemma}

\begin{proposition}[Localized virial identity,  \cite{TaoVisZha07, Din21}]\label{Virial identity}
Let $d = p = 3$, $\gamma > 0$, and $0 < \mu < 2$.
Given a suitable real-valued radial function $w \in C^\infty(\mathbb{R}^3)$ and the solution $u(t)$ to \eqref{NLS}, we define
\begin{align*}
	I(t)
		:= \int_{\mathbb{R}^3}w(x)|u(t,x)|^2dx.
\end{align*}
Then, we have
\begin{align*}
	I'(t)
		& = 2\text{Im}\int_{\mathbb{R}^3}\frac{w'(r)}{r}\overline{u(t,x)}x\cdot\nabla u(t,x)dx, \\
	I''(t)
		& = \int_{\mathbb{R}^3}F_1(w,r)|x\cdot\nabla u(t,x)|^2dx + 4\int_{\mathbb{R}^3}\frac{w'(r)}{r}|\nabla u(t,x)|^2dx - \int_{\mathbb{R}^3}F_2(w,r)|u(t,x)|^2dx \notag \\
		& \hspace{4.0cm} - \int_{\mathbb{R}^3}F_3(w,r)|u(t,x)|^4dx + 2\mu\int_{\mathbb{R}^3}w'(r)\frac{\gamma}{r^{\mu+1}}|u(t,x)|^2dx,
\end{align*}
where $r = |x|$,
\begin{gather*}
	F_1(w,r)
		:= 4\left\{\frac{w''(r)}{r^2} - \frac{w'(r)}{r^3}\right\},\ \ 
	F_2(w,r)
		:= w^{(4)}(r) + \frac{4}{r}w^{(3)}(r),\ \ 
	F_3(w,r)
		:= w''(r) + \frac{2}{r}w'(r).
\end{gather*}
\end{proposition}

\section{Parameter independence of the splitting below the ``radial'' ground state}\label{Sec:Parameter independence of the splitting below the ``radial'' ground state}

In this section, we prove that $\mathcal{R}_{\omega,\gamma}^{\alpha,\beta,\pm}$ are independent of $(\alpha,\beta)$, where $\mathcal{R}_{\omega,\gamma}^{\alpha,\beta,\pm}$ are defined as
\begin{gather*}
	\mathcal{R}_{\omega,\gamma}^{\alpha,\beta,+}
		:= \{\phi \in H_\text{rad}^1(\mathbb{R}^3) : S_{\omega,\gamma}(\phi) < r_{\omega,\gamma},\ K_{\omega,\gamma}^{\alpha,\beta}(\phi)\geq0\}, \\
	\mathcal{R}_{\omega,\gamma}^{\alpha,\beta,-}
		:= \{\phi \in H_\text{rad}^1(\mathbb{R}^3) : S_{\omega,\gamma}(\phi) < r_{\omega,\gamma},\ K_{\omega,\gamma}^{\alpha,\beta}(\phi) < 0\}.
\end{gather*}
The proof is based on the argument in \cite{IbrMasNak11}.

\begin{proposition}[Equivalence of $H^1$-norm and $S_{\omega,\gamma}$, \cite{HamIkeISAAC}]\label{Equivalence of H1 and S}
Let $d = p = 3$, $\gamma > 0$, $0 < \mu < 2$, and $\omega > 0$.
Let $(\alpha,\beta)$ satisfy \eqref{104}.
If $f \in H^1(\mathbb{R}^3)$ satisfies $K_{\omega,\gamma}^{\alpha,\beta}(f)\geq0$, then
\begin{align*}
	2(\alpha - \beta)S_{\omega,\gamma}(f)
		\leq (\alpha - \beta)\|f\|_{H_{\omega,\gamma}^1}^2
		\leq (4\alpha - 3\beta)S_{\omega,\gamma}(f)
\end{align*}
holds, where $\|f\|_{H_{\omega,\gamma}^1}^2 := \omega\|f\|_{L^2}^2 + \|(-\Delta_\gamma)^\frac{1}{2}f\|_{L^2}^2$.
\end{proposition}

\begin{lemma}[Positivity of $K_{\omega,\gamma}^{\alpha,\beta}$, \cite{HamIkeISAAC}]\label{Positivity of K}
Let $d = p = 3$, $\gamma > 0$, $0 < \mu < 2$, and $\omega > 0$.
Let $(\alpha,\beta)$ satisfy \eqref{104}.
Suppose that $\{f_n\}$ is a bounded sequence in $H^1(\mathbb{R}^3) \setminus \{0\}$ and satisfies $\|\nabla f_n\|_{L^2} \longrightarrow 0$ as $n \rightarrow \infty$.
Then, there exists $n_0 \in \mathbb{N}$ such that $K_{\omega,\gamma}^{\alpha,\beta}(f_n) > 0$ for any $n \geq n_0$.
\end{lemma}

\begin{lemma}[Rewriting of $r_{\omega,\gamma}^{\alpha,\beta}$, \cite{HamIkeISAAC}]\label{Rewriting of infimum}
Let $d = p = 3$, $\gamma > 0$, $0 < \mu < 2$, and $\omega > 0$.
Let $(\alpha,\beta)$ satisfy \eqref{104}.
If $f\in H_\text{rad}^1(\mathbb{R}^3) \setminus \{0\}$ satisfies $K_{\omega,\gamma}^{\alpha,\beta}(f) < 0$, then there exists $0 < \lambda < 1$ such that
\begin{align*}
	K_{\omega,\gamma}^{\alpha,\beta}(\lambda f)
		= 0\ \ \text{ and }\ \ 
	r_{\omega,\gamma}^{\alpha,\beta}
		\leq S_{\omega,\gamma}(\lambda f)
		= T_{\omega,\gamma}^{\alpha,\beta}(\lambda f)
		< T_{\omega,\gamma}^{\alpha,\beta}(f),
\end{align*}
where the functional $T_{\omega,\gamma}^{\alpha,\beta}$ is defined as
\begin{align*}
	T_{\omega,\gamma}^{\alpha,\beta}(f)
		:= S_{\omega,\gamma}(f) - \frac{1}{2\alpha - \beta}K_{\omega,\gamma}^{\alpha,\beta}(f).
\end{align*}
In particular,
\begin{align*}
	r_{\omega,\gamma}^{\alpha,\beta}
		= \inf\{T_{\omega,\gamma}^{\alpha,\beta}(f) : f \in H_\text{rad}^1(\mathbb{R}^3) \setminus \{0\},\ K_{\omega,\gamma}^{\alpha,\beta}(f) \leq 0\}
\end{align*}
holds.
\end{lemma}

\begin{proposition}\label{Initial data set}
Let $d = p = 3$, $\gamma > 0$, $0 < \mu < 2$, and $\omega > 0$.
Let $(\alpha,\beta)$ satisfy \eqref{104}.
Then, $\mathcal{R}_{\omega,\gamma}^{\alpha,\beta,\pm}$ are independent of $(\alpha,\beta)$.
\end{proposition}

\begin{proof}
First, we check that $\mathcal{R}_{\omega,\gamma}^{\alpha,\beta,\pm}$ are open subsets in $H^1(\mathbb{R}^3)$.
$\mathcal{R}_{\omega,\gamma}^{\alpha,\beta,-}$ is clearly open subset in $H^1(\mathbb{R}^3)$ and $\phi$ satisfying $K_{\omega,\gamma}^{\alpha,\beta}(\phi) > 0$ is clearly an interior point in $\mathcal{R}_{\omega,\gamma}^{\alpha,\beta,+}$ by the continuity of $S_{\omega,\gamma}$ and $K_{\omega,\gamma}^{\alpha,\beta}$.
From the definition of $r_{\omega,\gamma}$, a function $\phi\in \mathcal{R}_{\omega,\gamma}^{\alpha,\beta,+}$ satisfying $K_{\omega,\gamma}^{\alpha,\beta}(\phi) = 0$ is only $\phi = 0$, so it suffices to prove that $0$ is an interior point of $\mathcal{R}_{\omega,\gamma}^{\alpha,\beta,+}$.
It follows from Lemma \ref{Positivity of K}.\\
Next, we prove that $\mathcal{R}_{\omega,\gamma}^{\alpha,\beta,+}$ and $\mathcal{R}_{\omega,\gamma}^{\alpha,\beta,-}$ are independent of $(\alpha,\beta)$.
Let $(\alpha,\beta)$ satisfy \eqref{104} and $2\alpha - 3\beta \neq 0$.
We note $\mathcal{R}_{\omega,\gamma}^{\alpha,\beta,+} \cup \mathcal{R}_{\omega,\gamma}^{\alpha,\beta,-} = \{\phi\in H_\text{rad}^1(\mathbb{R}^3) : S_{\omega,\gamma}(\phi) < r_{\omega,\gamma}\}$ is independent of $(\alpha,\beta)$.
In addition, we have
\begin{itemize}
\item[(1)]
$K_{\omega,\gamma}^{\alpha,\beta}(e^{\alpha\lambda}\varphi(e^{\beta\lambda}\,\cdot\,)) > 0$ for any $\varphi \in \mathcal{R}_{\omega,\gamma}^{\alpha,\beta,+} \setminus \{0\}$ and $\lambda \leq 0$,
\item[(2)]
$S_{\omega,\gamma}(e^{\alpha\lambda}\varphi(e^{\beta\lambda}\,\cdot\,))$ does not increase as $0 \geq \lambda \rightarrow -\infty$ for any $\varphi \in \mathcal{R}_{\omega,\gamma}^{\alpha,\beta,+}$,
\item[(3)]
$e^{\alpha\lambda}\varphi(e^{\beta\lambda}\,\cdot\,) \longrightarrow 0$ in $H^1(\mathbb{R}^3)$ as $\lambda \rightarrow -\infty$,
\end{itemize}
which imply that $\mathcal{R}_{\omega,\gamma}^{\alpha,\beta,+}$ is connected.
Thus, each $\mathcal{R}_{\omega,\gamma}^{\alpha,\beta,+}$ can not be separated by $\mathcal{R}_{\omega,\gamma}^{\alpha',\beta',+}$ and $\mathcal{R}_{\omega,\gamma}^{\alpha',\beta',-}$ for any $(\alpha',\beta')\neq (\alpha,\beta)$ with \eqref{104} and $2\alpha'- 3\beta' \neq 0$.
$\mathcal{R}_{\omega,\gamma}^{\alpha,\beta,+}$ and $\mathcal{R}_{\omega,\gamma}^{\alpha',\beta',+}$ contain $0$, so $\mathcal{R}_{\omega,\gamma}^{\alpha,\beta,+}=\mathcal{R}_{\omega,\gamma}^{\alpha',\beta',+}$ holds.
Let $(\alpha,\beta)$ satisfy \eqref{104} and $2\alpha - 3\beta = 0$.
We take a sequence $\{(\alpha_n,\beta_n)\}$ satisfying \eqref{104}, $2\alpha_n - 3\beta_n > 0$, and $(\alpha_n,\beta_n) \longrightarrow (\alpha,\beta)$ as $n \rightarrow \infty$.
Then, $K_{\omega,\gamma}^{\alpha_n,\beta_n} \longrightarrow K_{\omega,\gamma}^{\alpha,\beta}$ implies
\begin{align*}
	\mathcal{R}_{\omega,\gamma}^{\alpha,\beta,\pm}
		\subset \bigcup_{n\in \mathbb{N}}\mathcal{R}_{\omega,\gamma}^{\alpha_n,\beta_n,\pm},
\end{align*}
where the double signs correspond.
Since $\mathcal{R}_{\omega,\gamma}^{\alpha_n,\beta_n,+}$ and $\mathcal{R}_{\omega,\gamma}^{\alpha_n,\beta_n,-}$ are independent of $(\alpha_n,\beta_n)$, so the left hand side also is independent of $(\alpha,\beta)$.
\end{proof}

\section{Well-posedness}\label{Sec:Well-posedness}

\subsection{Local well-posedness}\label{Subsec:Local well-posedness}

We define the following function spaces:
\begin{align*}
	S_\gamma^s(I)
		:= L_t^2(I;W_\gamma^{s,6}(\mathbb{R}^3)) \cap L_t^\infty(I;H_\gamma^s(\mathbb{R}^3)),\ \ \ 
	\dot{S}_\gamma^s(I)
		:= L_t^2(I;\dot{W}_\gamma^{s,6}(\mathbb{R}^3)) \cap L_t^\infty(I;\dot{H}_\gamma^s(\mathbb{R}^3)).
\end{align*}

We also define a solution to \eqref{NLS} clearly.

\begin{definition}[Solution]
Let $t_0 \in \mathbb{R}$ and $u_0 \in H_\gamma^1(\mathbb{R}^3)$.
Let a time interval $I$ contain $t_0$.
We call $u : I \times \mathbb{R}^3 \longrightarrow \mathbb{C}$ a solution to \eqref{NLS} with initial data $u(t_0) = u_0$ if $u \in C_t(K;H_\gamma^1(\mathbb{R}^3)) \cap S_\gamma^1(K)$ for any compact set $K \subset I$ and satisfies the Duhamel formula
\begin{align*}
	u(t)
		= e^{i(t-t_0)\Delta_\gamma}u_0 + i\int_{t_0}^te^{i(t-s)\Delta_\gamma}(|u|^2u)(s)ds
\end{align*}
for any $t \in I$.
We say that $I$ is the lifespan of $u$.
\end{definition}

\begin{theorem}[Local well-posedness]\label{Local well-posedness}
Let $d = p = 3$, $\gamma > 0$, and $0 < \mu < 2$.
For any $u_0 \in H^1(\mathbb{R}^3)$, there exists $T = T(\|u_0\|_{H^1}) > 0$ and a unique solution $u : (t_0 - T,t_0 + T) \times \mathbb{R}^3 \longrightarrow \mathbb{C}$ to \eqref{NLS} with initial data $u(t_0) = u_0$.
\end{theorem}

\begin{proof}
For simplicity, we set $t_0 = 0$ and $I = (-T,T)$.
We define a function space $E$ as
\begin{align*}
	E
		:= \{u \in C_t(I;H_\gamma^1(\mathbb{R}^3)) \cap S_\gamma^1(I) : \|u\|_{S_\gamma^1(I)} \leq 2c\,\|u_0\|_{H^1}\}
\end{align*}
and a metric $d$ on $E$ as
\begin{align*}
	d(u,v)
		:= \|u - v\|_{S_\gamma^0(I)}.
\end{align*}
We prove that
\begin{align*}
	\Phi_{u_0}(u)
		= e^{it\Delta_\gamma}u_0 + i\int_0^te^{i(t-s)\Delta_\gamma}(|u|^2u)(s)ds
\end{align*}
is a contraction map on $(E,d)$.
Using Theorem \ref{Strichartz estimate}, Lemmas \ref{Norm equivalence}, \ref{Fractional calculus}, and Sobolev embedding, we have
\begin{align*}
	\|\Phi_{u_0}(u)\|_{S_\gamma^1(I)}
		& \leq c\,\|(1-\Delta_\gamma)^\frac{1}{2} u_0\|_{L_x^2} + c\,\|(1-\Delta_\gamma)^\frac{1}{2}(|u|^2u)\|_{L_t^2(I;L_x^\frac{6}{5})} \\
		& \leq c\,\|u_0\|_{H^1} + c\,T^\frac{1}{10}\|u\|_{L_t^5(I;L_x^6)}^2 \|(1-\Delta_\gamma)^\frac{1}{2}u\|_{L_t^\infty(I; L_x^2)} \\
		& \leq c\,\|u_0\|_{H^1} + c\,T^\frac{1}{10} \|(1-\Delta_\gamma)^\frac{1}{2}u\|_{L_t^5(I;L_x^\frac{30}{11})}^2 \|(1-\Delta_\gamma)^\frac{1}{2}u\|_{L_t^\infty(I;L_x^2)} \\
		& \leq c\,\|u_0\|_{H^1} + c\,T^\frac{1}{10}\|u\|_{S_\gamma^1(I)}^3 \\
		& \leq \left\{1+(2c)^3T^\frac{1}{10}\|u_0\|_{H^1}^2\right\}c\,\|u_0\|_{H^1}
\end{align*}
and
\begin{align*}
	\|\Phi_{u_0}(u)-\Phi_{u_0}(v)\|_{S_\gamma^0(I)}
		\leq (2c)^3T^\frac{1}{10}\|u_0\|_{H^1}^2\|u-v\|_{S_\gamma^0(I)}.
\end{align*}
Here, we take $T > 0$ sufficiently small such as $(2c)^3T^\frac{1}{10}\|u_0\|_{H^1}^2 < 1$.
Then, $\Phi_{u_0}$ is a contraction map on $(E,d)$ and there exists a unique solution to \eqref{NLS} in $[-T,T]$.
\end{proof}

\begin{lemma}[\cite{HamIkeISAAC}]\label{Uniform estimate for K}
Let $d = p = 3$, $\gamma > 0$, and $0 < \mu < 2$, $u_0 \in H_\text{rad}^1(\mathbb{R}^3)$ and let $(\alpha,\beta)$ satisfy \eqref{104}.
We assume that
\begin{align*}
	S_{\omega,\gamma}(u_0)
		< r_{\omega,\gamma}\ \ \text{ and }\ \ 
	K_{\omega,\gamma}^{\alpha,\beta}(u_0)
		> 0\ \text{ for some }\ \omega > 0.
\end{align*}
Then, a solution $u$ to \eqref{NLS} with \eqref{IC} satisfies $K_{\omega,\gamma}^{\alpha,\beta}(u(t)) > 0$ for any $t \in (T_\text{min},T_\text{max})$.
\end{lemma}

From Lemma \ref{Uniform estimate for K}, the following global well-posedness holds immediately.

\begin{theorem}[Global well-posedness, \cite{HamIkeProceeding, HamIkeISAAC}]\label{Global well-posedness}
Let $d = p = 3$, $\gamma > 0$, and $0 < \mu < 2$.
Let $u_0 \in H_\text{rad}^1(\mathbb{R}^3)$ and let $(\alpha,\beta)$ satisfy \eqref{104}.
We assume that
\begin{align*}
	S_{\omega,\gamma}(u_0)
		< r_{\omega,\gamma}\ \ \text{ and }\ \ 
	K_{\omega,\gamma}^{\alpha,\beta}(u_0)
		> 0\ \text{ for some }\ \omega > 0.
\end{align*}
Then, a solution $u$ to \eqref{NLS} with \eqref{IC} exists globally in time.
\end{theorem}

Now, we get a lower bound of $K_\gamma$, which is strictly stronger than Lemma \ref{Uniform estimate for K}.

\begin{lemma}\label{Estimate of K from below}
Let $d = p = 3$, $\gamma > 0$, $0 < \mu < 2$, and $u_0 \in H_\text{rad}^1(\mathbb{R}^3)$.
We assume that
\begin{align*}
	S_{\omega,\gamma}(u_0)
		< r_{\omega,\gamma}\ \ \text{ and }\ \ 
	K_\gamma(u_0)
		> 0\ \text{ for some }\ \omega > 0.
\end{align*}
Then, it follows that
\begin{align*}
	K_\gamma(u(t))
		\geq \min\left\{r_{\omega,\gamma}-S_{\omega,\gamma}(u_0), \frac{2\mu}{7}\|(-\Delta_\gamma)^\frac{1}{2}u(t)\|_{L^2}^2\right\}
\end{align*}
for any $t \in (T_\text{min},T_\text{max})$.
\end{lemma}

\begin{proof}
We define a function
\begin{align*}
	J_{\omega,\gamma}(\lambda)
		= S_{\omega,\gamma}(e^{\alpha\lambda}u(e^{\beta\lambda}\,\cdot\,))
\end{align*}
We note that
\begin{align*}
	J_{\omega,\gamma}(0)
		= S_{\omega,\gamma}(u)\ \ \text{ and }\ \ 
	\frac{d}{d\lambda}J_{\omega,\gamma}(0)
		= K_\gamma(u),
\end{align*}
We solve $\frac{d}{d\lambda}J_{\omega,\gamma}(\lambda) = 0$.
This equation has only one solution by $K_\gamma(u) > 0$ and the solution is positive.
We set that the solution $\lambda = \lambda_0 > 0$.
We also solve $\frac{d^2}{d\lambda^2}J_{\omega,\gamma}(\lambda) + \frac{d}{d\lambda}J_{\omega,\gamma}(\lambda) = 0$.
This equation also has only one solution and we set that the solution $\lambda = \lambda_1$.
When $\lambda_1 < 0$, integrating $\frac{d^2}{d\lambda^2}J_{\omega,\gamma}(\lambda) + \frac{d}{d\lambda}J_{\omega,\gamma}(\lambda)$ over $[0,\lambda_0]$, we have
\begin{align*}
	0
		& > - K_\gamma(u) + J_{\omega,\gamma}(\lambda_0) - S_{\omega,\gamma}(u),
\end{align*}
where $\frac{d^2}{d\lambda^2}J_{\omega,\gamma}(\lambda) + \frac{d}{d\lambda}J_{\omega,\gamma}(\lambda) < 0$ on $[0,\lambda_0]$.
Since
\begin{align*}
	K_\gamma(e^{3\lambda_0}u(e^{2\lambda_0}\,\cdot\,))
		= \frac{d}{d\lambda}J_{\omega,\gamma}(\lambda_0)
		= 0,
\end{align*}
it follows that
\begin{align*}
	K_\gamma(u)
		> r_{\omega,\gamma} - S_{\omega,\gamma}(u).
\end{align*}
When $\lambda_1 \geq 0$, it follows from $\frac{d^2}{d\lambda^2}J_{\omega,\gamma}(\lambda_1) + \frac{d}{d\lambda}J_{\omega,\gamma}(\lambda_1) = 0$ that
\begin{align*}
	\frac{3}{2}\|u\|_{L^4}^4
		& \leq \frac{10}{7}\|\nabla u\|_{L^2}^2 + \frac{\mu(2\mu+1)}{7}\int_{\mathbb{R}^3}\frac{\gamma}{|x|^\mu}|u(x)|^2dx
		\leq \frac{10}{7}\|\nabla u\|_{L^2}^2 + \frac{5\mu}{7}\int_{\mathbb{R}^3}\frac{\gamma}{|x|^\mu}|u(x)|^2dx.
\end{align*}
Thus, we have
\begin{align*}
	K_\gamma(u)
		& \geq \frac{2}{7}\left\{2\|\nabla u\|_{L^2}^2 + \mu\int_{\mathbb{R}^3}\frac{\gamma}{|x|^\mu}|u(x)|^2dx\right\}
		\geq \frac{2\mu}{7}\|(-\Delta_\gamma)^\frac{1}{2} u\|_{L^2}^2.
\end{align*}
Therefore, using Proposition \ref{Equivalence of H1 and S}, we obtain
\begin{align*}
	K_\gamma(u(t))
		\geq \min\left\{r_{\omega,\gamma}-S_{\omega,\gamma}(u_0), \frac{2\mu}{7}\|(-\Delta_\gamma)^\frac{1}{2}u(t)\|_{L^2}^2\right\}
\end{align*}
for any $t\in (T_\text{min},T_\text{max})$.
\end{proof}

\subsection{Small data theory}\label{Subsec:Small data theory}

We define the following exponents:
\begin{align*}
	q_0
		= 5,\ \ \ 
	r_0
		= \frac{30}{11},\ \ \ 
	\rho
		= \frac{5}{3},\ \ \ 
	\kappa
		= \frac{30}{23}.
\end{align*}
$(q_0,r_0)$ and $(\rho',\kappa')$ are $L^2$-admissible pairs, $(q_0,q_0)$ is a $\dot{H}^\frac{1}{2}$-admissible pair, and these exponents satisfy
\begin{align*}
	\frac{1}{\rho}
		= \frac{2}{q_0} + \frac{1}{q_0},\ \ \ 
	\frac{1}{\kappa}
		= \frac{2}{q_0} + \frac{1}{r_0}.
\end{align*}
We also take exponents $q_1$, $r_1$, $q_2$, $r_2$, and $r_3$ as follows.
We choose
\begin{align*}
	q_1
		:= 4^+,\ \ \ 
	r_1
		:= 3^-,\ \ \ 
	q_2
		:= \infty^-,\ \ \ 
	r_2
		:= 2^+,\ \ \ 
	r_3
		:= 6^-
\end{align*}
satisfying that $(q_1,r_1)$ and $(q_2,r_2)$ are $L^2$-admissible pairs, the embedding $\dot{W}^{\frac{1}{2},r_1}(\mathbb{R}^3) \hookrightarrow L^{r_3}(\mathbb{R}^3)$ holds, $\|\,\cdot\,\|_{\dot{W}_\gamma^{\frac{1}{2},r_1}}$ and $\|\,\cdot\,\|_{\dot{W}^{\frac{1}{2},r_1}}$ are equivalent, $\|\,\cdot\,\|_{\dot{W}_\gamma^{1,r_2}}$ and $\|\,\cdot\,\|_{\dot{W}^{1,r_2}}$ are equivalent, and a nonlinear estimate
\begin{align*}
	\||u|^2u\|_{L_t^2W_x^{1,\frac{6}{5}}}
		\lesssim \|u\|_{L_t^{q_1}L_x^{r_3}}^2\|u\|_{L_t^{q_2}W_x^{1,r_2}}
\end{align*}
holds.
Here, $4^+$ is an arbitrarily preselected and fixed number larger than $4$.
$3^-$, $\infty^-$, $2^+$, and $6^-$ are defined similarly.

\begin{proposition}[Small data global existence]\label{Small data global existence}
Let $d = p = 3$, $\gamma > 0$, and $0 < \mu < 2$.
Let $t_0 \in \mathbb{R}$ and $u_0 \in H^\frac{1}{2}(\mathbb{R}^3)$.
There exists $\varepsilon_0 > 0$ such that for any $0 < \varepsilon < \varepsilon_0$, if
\begin{align*}
	\|e^{i(t-t_0)\Delta_\gamma}u_0\|_{L_t^{q_0}(t_0,\infty;L_x^{q_0})}
		< \varepsilon,
\end{align*}
then a solution to \eqref{NLS} with initial data $u(t_0) = u_0$ exists globally in time and satisfies
\begin{align*}
	\|u\|_{L_t^{q_0}(t_0,\infty;L_x^{q_0})}
		\lesssim \varepsilon.
\end{align*}
\end{proposition}

\begin{proof}
We set $A = \|u_0\|_{H^\frac{1}{2}}$ and $I = [t_0,\infty)$.
We define a space $E$ as
\begin{align*}
E := \left\{
\begin{array}{l}
\hspace{-0.2cm}u \in C_t(I;H_x^\frac{1}{2}(\mathbb{R}^3)) \cap L_t^{q_0}(I;W_x^{\frac{1}{2},r_0}(\mathbb{R}^3))
\end{array}
\hspace{-0.2cm}\left|
\begin{array}{l}
\|u\|_{L_t^\infty(I;H_x^\frac{1}{2})} \leq 2cA,\ \|u\|_{L_t^{q_0}(I;W_x^{\frac{1}{2},r_0})}\leq 2cA,\\[0.2cm]
\|u\|_{L_t^{q_0}(I;L_x^{q_0})} \leq 2\varepsilon
\end{array}
\right.\right\},
\end{align*}
and a distance $d$ on $E$ as
\begin{align*}
	d(u_1,u_2)
		= \|u_1-u_2\|_{L_t^{q_0}(I;L_x^{r_0})}.
\end{align*}
We prove that
\begin{align*}
	\Phi_{u_0}(u)
		= e^{i(t-t_0)\Delta_\gamma}u_0 + i\int_{t_0}^te^{i(t-s)\Delta_\gamma}(|u|^2u)(s)ds
\end{align*}
is a contraction map on $(E,d)$.
From Theorem \ref{Strichartz estimate} and Lemma \ref{Fractional calculus}, we get
\begin{align*}
	\|\Phi_{u_0}(u)\|_{L_t^{q_0}(I;L_x^{q_0})}
		& \leq \|e^{i(t-t_0)\Delta_\gamma}u_0\|_{L_t^{q_0}(I;L_x^{q_0})} + \left\|\int_{t_0}^te^{i(t-s)\Delta_\gamma}(|u|^{p-1}u)(s)ds\right\|_{L_t^{q_0}(I;L_x^{q_0})} \\
		& \leq \varepsilon + c\,\||u|^2u\|_{L_t^\rho(I;W_x^{\frac{1}{2},\kappa})} \\
		& \leq \varepsilon + c\,\|u\|_{L_t^{q_0}(I;L_x^{q_0})}^2\|u\|_{L_t^{q_0}(I;W_x^{\frac{1}{2},r_0})} \\
		& \leq \varepsilon + 8c^2A\varepsilon^2
		= (1 + 8c^2A\varepsilon)\varepsilon,
\end{align*}
\begin{align*}
	\|\Phi_{u_0}(u)\|_{L_t^\infty(I;H_x^\frac{1}{2}) \cap L_t^{q_0}(I;W_x^{\frac{1}{2},r_0})}
		\leq c\,\|u_0\|_{H^\frac{1}{2}} + c\,\||u|^2u\|_{L_t^\rho(I;W_x^{\frac{1}{2},\kappa})}
		\leq (1 + 8c\varepsilon^2)cA,
\end{align*}
and
\begin{align*}
	\|\Phi_{u_0}(u) - \Phi_{u_0}(v)\|_{L_t^{q_0}(I;L_x^{r_0})}
		& \leq \left\|\int_{t_0}^t e^{i(t-s)\Delta_\gamma}(|u|^2u - |v|^2v)(s)ds\right\|_{L_t^{q_0}(I;L_x^{r_0})} \\
		& \leq c\,\||u|^2u - |v|^2v\|_{L_t^\rho(I;L_x^\kappa)} \\
		& \leq c\,\left\{\|u\|_{L_t^{q_0}(I;L_x^{q_0})}^2 + \|v\|_{L_t^{q_0}(I;L_x^{q_0})}^2\right\}\|u - v\|_{L_t^{q_0}(I;L_x^{r_0})} \\
		& \leq 8c\,\varepsilon^2\|u - v\|_{L_t^{q_0}(I;L_x^{r_0})}.
\end{align*}
If we take a positive constant $\varepsilon$ satisfying $8c^2A\varepsilon \leq 1$ and $8c\varepsilon^2 \leq 1$, then $\Phi_{u_0}$ is a contraction map on $(E,d)$.
\end{proof}

\begin{proposition}[Persistence of regularity]\label{Persistence of regularity}
Let $d = p = 3$, $\gamma > 0$, and $0 < \mu < 2$.
Let $E > 0$, $L > 0$, and $I$ be a time interval.
Assume that $u : I \times \mathbb{R}^3 \longrightarrow \mathbb{C}$ is a solution to \eqref{NLS} satisfying
\begin{align*}
	\|u(t_0)\|_{H^1}
		\leq E\ \text{ for some }\ t_0 \in I\ \ \text{ and }\ \ 
	\|u\|_{L_t^{q_0}(I;L_x^{q_0})}
		\leq L.
\end{align*}
Then, we have
\begin{align*}
	\|u\|_{S_\gamma^1(I)}
		\lesssim_{E,L} 1.
\end{align*}
\end{proposition}

\begin{proof}
We assume $t_0 := \inf I$ without loss of generality.
We divide $I = \bigcup_{j=1}^J[t_{j-1}^\circ,t_j^\circ)$ to satisfy $t_0^\circ = t_0$ and
\begin{align*}
	c\,\|u\|_{L_t^{q_0}(t_{j-1}^\circ,t_j^\circ;L_x^{q_0})}^2
		\leq \frac{1}{2}.
\end{align*}
This inequality deduces
\begin{align*}
	\|u\|_{L_t^{q_0}(t_{j-1}^\circ,t_j^\circ;W_x^{\frac{1}{2},r_0})}
		& \leq c\,\|u(t_{j-1}^\circ)\|_{H^\frac{1}{2}} + c\,\|u\|_{L_t^{q_0}(t_{j-1}^\circ,t_j^\circ;L_x^{q_0})}^2\|u\|_{L_t^{q_0}(t_{j-1}^\circ,t_j^\circ;W_x^{\frac{1}{2},r_0})} \\
		& \leq c\,\|u(t_{j-1}^\circ)\|_{H^\frac{1}{2}} + \frac{1}{2}\|u\|_{L_t^{q_0}(t_{j-1}^\circ,t_j^\circ;W_x^{\frac{1}{2},r_0})} \\
		& \leq 2c\,\|u(t_{j-1}^\circ)\|_{H^\frac{1}{2}}
\end{align*}
and hence, we have
\begin{align*}
	\|u\|_{L_t^{q_0}(I;W_x^{\frac{1}{2},r_0})}
		\leq \sum_{j=1}^J\|u\|_{L_t^{q_0}(t_{j-1}^\circ,t_j^\circ;W_x^{\frac{1}{2},r_0})}
		\leq 2c\sum_{j=1}^J\|u(t_{j-1}^\circ)\|_{H^\frac{1}{2}}
		< \infty.
\end{align*}
Therefore, we obtain
\begin{gather*}
	\|u\|_{L_t^\infty(I;H_x^\frac{1}{2}) \cap L_t^{q_1}(I;W_x^{\frac{1}{2},r_1})}
		\leq c\,\|u(t_0)\|_{H^\frac{1}{2}} + c\,\|u\|_{L_t^{q_0}(I;L_x^{q_0})}^2\|u\|_{L_t^{q_0}(I;W_x^{\frac{1}{2},r_0})}
		< \infty.
\end{gather*}
Thus, we have
\begin{align*}
	u
		\in L_t^\infty(I;H_x^\frac{1}{2}(\mathbb{R}^3)) \cap L_t^{q_0}(I;W_x^{\frac{1}{2},r_0}(\mathbb{R}^3)) \cap L_t^{q_1}(I;W_x^{\frac{1}{2},r_1}(\mathbb{R}^3)).
\end{align*}
We divide $I = \bigcup_{j=1}^J[t_{j-1},t_j)$ to satisfy
\begin{align*}
	c\,\|u\|_{L_t^{q_1}(t_{j-1},t_j;\dot{W}_x^{\frac{1}{2},r_1})}^2
		\leq \frac{1}{2}.
\end{align*}
Then, it follows that
\begin{align*}
	\|u\|_{L_t^{q_2}(t_{j-1},t_j;W_x^{1,r_2})}
		& \leq c\,\|u(t_{j-1})\|_{H^1} + c\,\||u|^2u\|_{L_t^2(t_{j-1},t_j;W_x^{1,\frac{6}{5}})} \\
		& \leq c\,\|u(t_{j-1})\|_{H^1} + c\,\|u\|_{L_t^{q_1}(t_{j-1},t_j;L_x^{r_3})}^2\|u\|_{L_t^{q_2}(t_{j-1},t_j;W_x^{1,r_2})} \\
		& \leq c\,\|u(t_{j-1})\|_{H^1} + c\,\|u\|_{L_t^{q_1}(t_{j-1},t_j;\dot{W}_x^{\frac{1}{2},r_1})}^2\|u\|_{L_t^{q_2}(t_{j-1},t_j;W_x^{1,r_2})} \\
		& \leq c\,\|u(t_{j-1})\|_{H^1} + \frac{1}{2}\|u\|_{L_t^{q_2}(t_{j-1},t_j;W_x^{1,r_2})} \\
		& \leq 2c\,\|u(t_{j-1})\|_{H^1}
\end{align*}
and hence, we have
\begin{align*}
	\|u\|_{L_t^\infty(t_{j-1},t_j;H_x^1)}
		\leq c\,\|u(t_{j-1})\|_{H^1} + \frac{1}{2}\|u\|_{L_t^{q_2}(t_{j-1},t_j;W_x^{1,r_2})}
		\leq 2c\,\|u(t_{j-1})\|_{H^1}.
\end{align*}
Since $\|u(t)\|_{H_x^1}$ is continuous at $t_j$,
\begin{align*}
	\|u(t_j)\|_{H^1}
		\leq 2c\,\|u(t_{j-1})\|_{H^1}
		\leq \ldots
		\leq (2c)^j\|u(t_0)\|_{H^1}
\end{align*}
holds.
Therefore, we obtain
\begin{align*}
	\|u\|_{L_t^{q_2}(I;W_x^{1,r_2})}
		\leq \sum_{j=1}^J\|u\|_{L_t^{q_2}(t_{j-1},t_j;W_x^{1,r_2})}
		\leq 2c\sum_{j=1}^J\|u(t_{j-1})\|_{H^1}
		\leq \sum_{j=1}^J(2c)^j\|u(t_0)\|_{H^1}.
\end{align*}
That is, we get $u \in L_t^{q_2}(I;W_x^{1,r_2}(\mathbb{R}^3))$.
Therefore, we obtain
\begin{align*}
	\|u\|_{S_\gamma^1(I)}
		\leq c\,\|u(t_0)\|_{H^1} + c\,\|u\|_{L_t^{q_1}(I;\dot{W}_x^{\frac{1}{2},r_1})}^2\|u\|_{L_t^{q_2}(I;W_x^{1,r_2})}
		< \infty.
\end{align*}
\end{proof}

We can get the following proposition immediately from Proposition \ref{Persistence of regularity}.

\begin{proposition}[Small data scattering]\label{Small data scattering}
Let $d = p = 3$, $\gamma > 0$, and $0 < \mu < 2$.
We assume that a solution $u \in L_t^\infty([0,\infty);H_x^1(\mathbb{R}^3))$ to \eqref{NLS} satisfies
\begin{align*}
	\|u\|_{L_t^\infty(t_0,\infty;H_x^1)}
		\leq E
\end{align*}
for some $t_0 > 0$.
Then, there exists $\varepsilon > 0$ such that if
\begin{align*}
	\|e^{i(t-t_0)\Delta_\gamma}u(t_0)\|_{L_t^{q_0}(t_0,\infty;L_x^{q_0})}
		< \varepsilon,
\end{align*}
then $u$ scatters in positive time.
\end{proposition}

\begin{proof}
We take $\varepsilon > 0$ as in Proposition \ref{Small data global existence} with $A = E$.
From Proposition \ref{Small data global existence}, the unique solution $u$ to \eqref{NLS} satisfies
\begin{align*}
	\|u\|_{L_t^{q_0}(t_0,\infty;W_x^{\frac{1}{2},r_0})}
		\leq 2CE\ \ \text{ and }\ \ 
	\|u\|_{L_t^{q_0}(t_0,\infty;L_x^{q_0})}
		\leq 2\varepsilon.
\end{align*}
In addition, we have
\begin{align*}
	u
		\in L_t^{q_1}(t_0,\infty;W_x^{\frac{1}{2},r_1}(\mathbb{R}^3)) \cap L_t^{q_2}(t_0,\infty;W_x^{1,r_2}(\mathbb{R}^3))
\end{align*}
by using Proposition \ref{Persistence of regularity}.
Therefore, we obtain
\begin{align*}
	\|e^{-it\Delta_\gamma}u(t) - e^{-i\tau\Delta_\gamma}u(\tau)\|_{H_x^1}
		& = \left\|\int_\tau^te^{-is\Delta_\gamma}(|u|^2u)(s)ds\right\|_{H_x^1}
		\leq c\,\||u|^2u\|_{L_t^2(\tau,t;W_x^{1,\frac{6}{5}})} \\
		& \leq c\,\|u\|_{L_t^{q_1}(\tau,t;W_x^{\frac{1}{2},r_1})}^2\|u\|_{L_t^{q_2}(\tau,t;W_x^{1,r_2})}
		\longrightarrow 0\ \ \text{ as }\ \ t > \tau \rightarrow \infty,
\end{align*}
which implies that $\{e^{-it\Delta_\gamma}u(t)\}$ is a Cauchy sequence in $H^1(\mathbb{R}^3)$.
\end{proof}

\subsection{Stability}\label{Subsec:Stability}

In this subsection, we prove a stability result for scattering.

\begin{theorem}[Short time perturbation]\label{Short time perturbation}
Let $d = p = 3$, $\gamma > 0$, and $0 < \mu < 2$.
Let $u_0 \in H^1(\mathbb{R}^3)$, $t_0 \in \mathbb{R}$, $E > 0$, and $I(\ni t_0)$ be a time interval.
We assume that an approximate solution $\~{u}$ to \eqref{NLS} on $I \times \mathbb{R}^3$ satisfies
\begin{align}
	i\partial_t\~{u} + \Delta_\gamma\~{u}
		= - |\~{u}|^2\~{u} + e \label{116}
\end{align}
for some function $e$ and
\begin{align*}
	\|u_0\|_{H^1} + \|\~{u}(t_0)\|_{H^1}
		\leq E.
\end{align*}
There exists $\varepsilon_0 > 0$ such that, for any $0 < \varepsilon < \varepsilon_0$ if
\begin{align*}
	\|\~{u}\|_{L_t^{q_0}(I;L_x^{q_0})} + \|\~{u}\|_{L_t^{q_0}(I;\dot{W}_x^{\frac{1}{2},r_0})}
		\leq \varepsilon_0\ \ \text{ and }\ \ 
	\|u_0 - \~{u}(t_0)\|_{\dot{H}^\frac{1}{2}} + \||\nabla|^\frac{1}{2}e\|_{L_t^\rho(I;L_x^\kappa)}
		\leq \varepsilon,
\end{align*}
then \eqref{NLS} with initial data $u(t_0) = u_0$ has a unique solution $u$ on $I \times \mathbb{R}^3$ satisfying
\begin{align*}
	\||u|^2u - |\~{u}|^2\~{u}\|_{L_t^\rho(I;\dot{W}_x^{\frac{1}{2},\kappa})} + \|u - \~{u}\|_{\dot{S}_\gamma^\frac{1}{2}(I)}
		\lesssim_{E}\varepsilon.
\end{align*}
\end{theorem}

\begin{proof}
We define a function
\begin{align*}
	w(t)
		:= u(t) - \widetilde{u}(t)
		= e^{i(t-t_0)\Delta_\gamma}w(t_0) + i\int_{t_0}^te^{i(t-s)\Delta_\gamma}(|u|^2u - |\widetilde{u}|^2\~{u} + e)(s)ds.
\end{align*}
We may assume $t_0 = \inf I$ by the time symmetry.
We define
\begin{align*}
	S(T)
		:= \||u|^2u - |\~{u}|^2\~{u}\|_{L_t^\rho(t_0,T;\dot{W}_x^{\frac{1}{2},\kappa})}.
\end{align*}
It follows that
\begin{align*}
	\|w\|_{L_t^{q_0}(t_0,T;L_x^{q_0}) \cap L_t^{q_0}(t_0,T;\dot{W}_x^{\frac{1}{2},r_0})} 
		& \lesssim \|w(t_0)\|_{\dot{H}^\frac{1}{2}} + \||u|^2u - |\~{u}|^2\~{u}\|_{L_t^\rho(t_0,T;\dot{W}_x^{\frac{1}{2},\kappa})} + \||\nabla|^\frac{1}{2}e\|_{L_t^\rho(t_0,T;L_x^\kappa)} \\
		& \leq \varepsilon + S(T).
\end{align*}
Then, we have
\begin{align*}
	S(T)
		& \lesssim \left\{\|u\|_{L_t^{q_0}(t_0,T;L_x^{q_0})} + \|\~{u}\|_{L_t^{q_0}(t_0,T;L_x^{q_0})}\right\}\left\{\|u\|_{L_t^{q_0}(t_0,T;\dot{W}_x^{\frac{1}{2},r_0})} + \|\~{u}\|_{L_t^{q_0}(t_0,T;\dot{W}_x^{\frac{1}{2},r_0})}\right\}\|w\|_{L_t^{q_0}(t_0,T;L_x^{q_0})} \\
		& \hspace{1.0cm} + \left\{\|u\|_{L_t^{q_0}(t_0,T;L_x^{q_0})}^2 + \|\~{u}\|_{L_t^{q_0}(t_0,T;L_x^{q_0})}^2\right\}\|w\|_{L_t^{q_0}(t_0,T;\dot{W}_x^{\frac{1}{2},r_0})} \\
		& \lesssim \left\{\|w\|_{L_t^{q_0}(t_0,T;L_x^{q_0})} + \|\~{u}\|_{L_t^{q_0}(t_0,T;L_x^{q_0})}\right\}\left\{\|w\|_{L_t^{q_0}(t_0,T;\dot{W}_x^{\frac{1}{2},r_0})} + \|\~{u}\|_{L_t^{q_0}(t_0,T;\dot{W}_x^{\frac{1}{2},r_0})}\right\}\|w\|_{L_t^{q_0}(t_0,T;L_x^{q_0})} \\
		& \hspace{1.0cm} + \left\{\|w\|_{L_t^{q_0}(t_0,T;L_x^{q_0})}^2 + \|\~{u}\|_{L_t^{q_0}(t_0,T;L_x^{q_0})}^2\right\}\|w\|_{L_t^{q_0}(t_0,T;\dot{W}_x^{\frac{1}{2},r_0})} \\
		& \lesssim \left[\{\varepsilon + S(T)\} + \varepsilon_0\right]^2\{\varepsilon + S(T)\} + \left[\{\varepsilon + S(T)\}^2 + \varepsilon_0^2\right]\{\varepsilon + S(T)\} \\
		& \lesssim \{\varepsilon + S(T)\}^3 + 2\{\varepsilon + S(T)\}^2\varepsilon_0 + \{\varepsilon + S(T)\}\varepsilon_0^2 \\
		& \lesssim \{\varepsilon + S(T)\}^3 + \{\varepsilon + S(T)\}\varepsilon_0^2.
\end{align*}
If we take $\varepsilon_0$ sufficiently small,
\begin{align*}
	\varepsilon + S(T)
		\lesssim \{\varepsilon + S(T)\}^3 + \varepsilon.
\end{align*}
Also, if we set $\varepsilon_0$ sufficiently small, then $S(T) \lesssim \varepsilon$.
Since the right hand side is independent of $T$, we have
\begin{gather*}
	\||u|^2u - |\~{u}|^2\~{u}\|_{L_t^\rho(I;\dot{W}_x^{\frac{1}{2},\kappa})}
		\lesssim \varepsilon, \\
	\|w\|_{\dot{S}_\gamma^\frac{1}{2}(I)}
		\lesssim \|w(t_0)\|_{\dot{H}^\frac{1}{2}} + \||\nabla|^\frac{1}{2}(|u|^2u - |\~{u}|^2\~{u})\|_{L_t^\rho(I;L_x^\kappa)} + \||\nabla|^\frac{1}{2}e\|_{L_t^\rho(I;L_x^\kappa)}
		\lesssim \varepsilon.
\end{gather*}
\end{proof}

\begin{theorem}[Long time perturbation]\label{Long time perturbation}
Let $d = p = 3$, $\gamma > 0$, and $0 < \mu < 2$.
Let $u_0 \in H^1(\mathbb{R}^3)$, $t_0 \in \mathbb{R}$, and $I(\ni t_0)$ be a time interval.
We assume that an approximate solution $\~{u}$ to \eqref{NLS} on $I \times \mathbb{R}^3$ satisfies \eqref{116} for some function $e$ and $\~{u}$ and $u_0$ satisfy
\begin{align*}
	\|u_0\|_{H^1} + \|\~{u}(t_0)\|_{H^1}
		\leq E\ \ \text{ and }\ \ 
	\|\~{u}\|_{L_t^{q_0}(I;L_x^{q_0})}
		\leq L.
\end{align*}
for some $E, L > 0$.
There exists $\varepsilon_1 = \varepsilon_1(E,L) > 0$ such that, for any $0 < \varepsilon < \varepsilon_1$ if
\begin{align*}
	\|u_0 - \~{u}(t_0)\|_{\dot{H}^\frac{1}{2}} + \||\nabla|^\frac{1}{2}e\|_{L_t^\rho(I;L_x^\kappa)}
		\leq \varepsilon,
\end{align*}
then \eqref{NLS} with initial data $u(t_0) = u_0$ has a unique solution $u : I \times \mathbb{R}^3 \longrightarrow \mathbb{C}$ satisfying
\begin{align*}
	\|u - \~{u}\|_{\dot{S}_\gamma^\frac{1}{2}(I)}
		\lesssim_{E,L}\varepsilon.
\end{align*}
\end{theorem}

\begin{proof}
Let $w$ be the same function as in Theorem \ref{Short time perturbation}.
We may assume $t_0 = \inf I$ by the time symmetry.
From the same argument with Proposition \ref{Persistence of regularity}, we have
\begin{align*}
	\|\~{u}\|_{L_t^{q_0}(I_j;\dot{W}_x^{\frac{1}{2},r_0})}
		\lesssim_{E,L}1.
\end{align*}
Take the positive constant $\varepsilon_0$ given in Theorem \ref{Short time perturbation}.
We divide the interval $I$ satisfying $I = \cup_{j=1}^JI_j = \cup_{j=1}^J[t_{j-1},t_j)$ and $\|\~{u}\|_{L_t^{q_0}(I_j;L_x^{q_0})} + \|\~{u}\|_{L_t^{q_0}(I_j;\dot{W}_x^{\frac{1}{2},r_0})} \leq \varepsilon_0$ for each $1 \leq j \leq J$.
Put
\begin{align*}
	\kappa_j
		:= \||u|^2u - |\~{u}|^2\~{u}\|_{L_t^\rho(I_j;\dot{W}_x^{\frac{1}{2},\kappa})}.
\end{align*}
Then, it follows from Theorem \ref{Short time perturbation} that if
\begin{align}
	\|w(t_j)\|_{\dot{H}^\frac{1}{2}} + \||\nabla|^\frac{1}{2}e\|_{L_t^\rho(I_j;L_x^\kappa)}
		\leq \eta_j \label{142}
\end{align}
for $0 < \eta_j < \varepsilon_0$, then
\begin{align*}
	\kappa_j + \|w\|_{\dot{S}_\gamma^\frac{1}{2}(I_j)}
		\leq C_0(E,L)\eta_j.
\end{align*}
Using the inequality
\begin{align*}
	\|w(t_j)\|_{\dot{H}^\frac{1}{2}}
		& \leq \|w(t_0)\|_{\dot{H}^\frac{1}{2}} + \left\|\int_{t_0}^{t_j}e^{i(t_j - s)\Delta_\gamma}(|u|^2u - |\~{u}|^2\~{u} + e)(s)ds\right\|_{\dot{H}^\frac{1}{2}} \\
		& \leq \|w(t_0)\|_{\dot{H}^\frac{1}{2}} + \left\|\int_{t_0}^te^{i(t - s)\Delta_\gamma}(|u|^2u - |\~{u}|^2\~{u} + e)(s)ds\right\|_{L_t^\infty(t_0,t_j;\dot{H}_x^\frac{1}{2})} \\
		& \leq \|w(t_0)\|_{\dot{H}^\frac{1}{2}} + c\,\||u|^2u - |\~{u}|^2\~{u}\|_{L_t^\rho(t_0,t_j;\dot{W}_x^{\frac{1}{2},r_0})} + c\,\||\nabla|^\frac{1}{2}e\|_{L_t^\rho(t_0,t_j;L_x^\kappa)},
\end{align*}
the left hand side of \eqref{142} is estimated as follows:
\begin{align*}
	\|w(t_j)\|_{\dot{H}^\frac{1}{2}} + \||\nabla|^\frac{1}{2}e\|_{L_t^\rho(t_0,t_j;L_x^\kappa)}
		\leq C_1\varepsilon + C_1\sum_{\ell=1}^{j-1}\kappa_\ell
\end{align*}
for some $C_1 > 1$.
We take a constant $\alpha \geq \max\{2,2C_0C_1\}$.
For $\varepsilon' \leq \varepsilon_0$, we set
\begin{align*}
	\eta_j
		= \eta_j(\varepsilon')
		:= \alpha^{j-J-1}\varepsilon'
\end{align*}
for $j \in [1,J+1]$ and $\varepsilon \leq \frac{1}{C_1}\eta_1$.
Then, we have
\begin{align*}
	\eta_1
		< \eta_2
		< \ldots
		< \eta_{J+1}
		= \varepsilon'
		\leq \varepsilon_0.
\end{align*}
For such $\eta_j$, we prove \eqref{142} by using induction.
We can see \eqref{142} with $j = 0$ by the assumption of this theorem.
Let $1 \leq k \leq J-1$.
To prove \eqref{142} with $j = k+1$, we assume that \eqref{142} holds for each $0 \leq j \leq k$.
We see
\begin{align*}
	C_1\varepsilon
		\leq \eta_1
		= \alpha^{-k}\eta_{k+1}
\end{align*}
and
\begin{align*}
	C_1\kappa_\ell
		\leq C_1C_0\eta_\ell
		\leq \frac{1}{2}\alpha\eta_\ell
		= \frac{1}{2}\alpha^{\ell-k}\eta_{k+1}
\end{align*}
for $\ell \in [1,k]$.
These inequalities deduce
\begin{align*}
	C_1\varepsilon + C_1\sum_{\ell=1}^k\kappa_\ell
		\leq \alpha^{-k}\left(1 + \frac{1}{2}\sum_{\ell=1}^k\alpha^\ell\right)\eta_{k+1}
		\leq \eta_{k+1}.
\end{align*}
Set $\varepsilon_1 := \frac{1}{C_1}\eta_1(\varepsilon_0)$ and assume $\varepsilon \leq \varepsilon_1$.
We define $\varepsilon' > 0$ as $\varepsilon = \frac{1}{C_1}\eta_1(\varepsilon')$.
Then, we notice $\varepsilon' \leq \varepsilon_0$ and $\varepsilon' = C_1\alpha^J\varepsilon$.
Therefore, we obtain
\begin{align*}
	\|w\|_{\dot{S}_\gamma^\frac{1}{2}(I)}
		& \leq \sum_{j=1}^J\|w\|_{\dot{S}_\gamma^\frac{1}{2}(I_j)}
		\leq \sum_{j=1}^JC_0\eta_j(\varepsilon')
		= \frac{1}{C_1}\sum_{j=1}^JC_0C_1\eta_j(\varepsilon') \\
		& \leq \frac{1}{C_1}\sum_{j=1}^J\frac{1}{2}\alpha \cdot \alpha^{j-J-1}\varepsilon'
		= \frac{\varepsilon'}{2C_1} \cdot \frac{\alpha^{1-J}(\alpha^J - 1)}{\alpha - 1}
		\leq \frac{1}{C_1}\varepsilon'
		= \alpha^J\varepsilon.
\end{align*}
\end{proof}

\subsection{Final state problem}\label{Subsec:Final state problem}

In this subsection, we consider a final state problem and prove existence of wave operators.

\begin{lemma}\label{Time decay}
Let $d = 3$, $1 < p < 5$, $\gamma > 0$, and $0 < \mu < 2$.
If $\{t_n\} \subset \mathbb{R}$ satisfies $|t_n| \longrightarrow \infty$ as $n \rightarrow \infty$, then we have
\begin{align*}
	e^{it_n\Delta_\gamma}f
		\longrightarrow 0\ \text{ in }\ L^{p+1}(\mathbb{R}^3)\ \ \text{ and }\ \ 
	e^{it_n\Delta_\gamma}f
		\xrightharpoonup[]{\hspace{0.4cm}}0\ \text{ in }\ H_\gamma^1(\mathbb{R}^3)
\end{align*}
as $n \rightarrow \infty$ for any $f \in H^1(\mathbb{R}^3)$.
\end{lemma}

\begin{proof}
By the density argument, we assume $f \in C_c^\infty(\mathbb{R}^3)$.\\
We prove $e^{it_n\Delta_\gamma}f \longrightarrow 0$ in $L^{p+1}(\mathbb{R}^3)$ as $n \rightarrow \infty$.
Let $F(t) := \|e^{it\Delta_\gamma}f\|_{L_x^6}$.
By Theorem \ref{Strichartz estimate},  we have $F \in L_t^2(\mathbb{R})$.
Since it follows from Sobolev embedding that
\begin{align*}
	|\partial_tF|
		\leq \|\partial_t[e^{it\Delta_\gamma}f]\|_{L_x^6}
		\lesssim \|\Delta_\gamma f\|_{\dot{H}_x^1}
		\lesssim 1,
\end{align*}
$F$ is Lipschitz.
Therefore, $F(t_n) \longrightarrow 0$ as $n \rightarrow \infty$.
Using the interpolation,
\begin{align*}
	\|e^{it_n\Delta_\gamma}f\|_{L^{p+1}}
		\leq \|e^{it_n\Delta_\gamma}f\|_{L^2}^{1-\theta}\|e^{it_n\Delta_\gamma}f\|_{L^6}^\theta
		= \|f\|_{L^2}^{1-\theta}\|e^{it_n\Delta_\gamma}f\|_{L^6}^\theta
\end{align*}
for $\theta \in (0,1)$ with $\frac{1}{p+1} = \frac{1-\theta}{2} + \frac{\theta}{6}$, which implies the desired result.\\
Next, we prove $e^{it_n\Delta_\gamma}f \xrightharpoonup[]{\hspace{0.4cm}}0$ in $H_\gamma^1(\mathbb{R}^3)$ as $n \rightarrow \infty$.
We take any $g \in C_c^\infty(\mathbb{R}^3)$.
Then, we have
\begin{align*}
	\<e^{it_n\Delta_\gamma}f,g\>_{H_\gamma^1}
		\leq \|e^{it_n\Delta_\gamma}(1-\Delta_\gamma)^\frac{1}{2}f\|_{L^3}\|(1-\Delta_\gamma)^\frac{1}{2}g\|_{L^\frac{3}{2}}
		\longrightarrow 0
\end{align*}
as $n \rightarrow \infty$.
\end{proof}

\begin{lemma}[Existence of wave operators]\label{Existence of wave operators}
Let $d = p = 3$, $\gamma > 0$, and $0 < \mu < 2$.
Let $Q_{\omega,\gamma}$ be the ``radial'' ground state to \eqref{SP}.
Suppose that $u_+ \in H_\text{rad}^1(\mathbb{R}^3)$ satisfies
\begin{align*}
	\frac{\omega}{2}M[u_+] + K[u_+]
		< S_{\omega,\gamma}(Q_{\omega,\gamma})\ \text{ for some }\ \omega > 0,
\end{align*}
where the functional $K$ is defined as $K[f] := \frac{1}{2}\|(-\Delta_\gamma)^\frac{1}{2}f\|_{L^2}^2$.
Then, there exists $u_0 \in H_\text{rad}^1(\mathbb{R}^3)$ such that the solution $u$ to \eqref{NLS} with \eqref{IC} exists globally in time and satisfies
\begin{gather*}
	S_{\omega,\gamma}(u_0)
		< S_{\omega,\gamma}(Q_{\omega,\gamma}),\ \ \ 
	K_\gamma(u_0)
		\geq 0, \\
	\|u\|_{L_t^{q_0}(T,\infty;L_x^{q_0})}
		\leq 2\|e^{it\Delta_\gamma}u_+\|_{L_t^{q_0}(T,\infty;L_x^{q_0})}\ \text{ for some }T > 0,
\end{gather*}
and
\begin{align*}
	\lim_{t \rightarrow \infty}\|u(t) - e^{it\Delta_\gamma}u_+\|_{H^1}
		= 0.
\end{align*}
\end{lemma}

\begin{proof}
When $u_+ \equiv 0$, we can see this lemma by taking $u_0 \equiv 0$.
So, we assume $u_+ \not\equiv 0$.
By $u_+ \in H^1(\mathbb{R}^3)$, there exists sufficiently large $T > 0$ such that
\begin{align*}
	\|e^{it\Delta_\gamma}u_+\|_{L_t^{q_0}(T,\infty;L_x^{q_0})}
		\leq \varepsilon_0,
\end{align*}
where $\varepsilon_0$ is given in Proposition \ref{Small data global existence}.
For simplicity, we set $I := [T,\infty)$.
We consider the following integral equation:
\begin{align*}
	u(t)
		= e^{it\Delta_\gamma}u_+ + i\int_t^\infty e^{i(t-s)\Delta_\gamma}(|u|^2u)(s)ds.
\end{align*}
We define a set
\begin{align*}
E := \left\{u \in C_t(I;H_x^\frac{1}{2}(\mathbb{R}^3)) \cap L_t^{q_0}(I;W_x^{\frac{1}{2},r_0}(\mathbb{R}^3)) \left|
\begin{array}{l}
\|u\|_{L_t^\infty(I;H_x^\frac{1}{2})} \leq 2C\varepsilon_0,\ \|u\|_{L_t^{q_0}(I;W_x^{\frac{1}{2},r_0})} \leq 2C\varepsilon_0,\\[0.2cm]
\|u\|_{L_t^{q_0}(I;L_x^{q_0})} \leq 2\|e^{it\Delta_\gamma}u_+\|_{L_t^{q_0}(I;L_x^{q_0})}
\end{array}
\right.\right\},
\end{align*}
and a distance on $E$
\begin{align*}
	d(u,v)
		:= \|u-v\|_{L_t^{q_0}(I;L_x^{r_0})}.
\end{align*}
By the same argument with Proposition \ref{Small data global existence}, there exists a unique solution $u$ to \eqref{NLS} on $(E,d)$.
From Proposition \ref{Persistence of regularity}, we have $\|u\|_{S_\gamma^1(I)} < \infty$.
Therefore, it follows that
\begin{align*}
	\|u(t) - e^{it\Delta_\gamma}u_+\|_{H_x^1}
		& = \left\|\int_t^\infty e^{i(t-s)\Delta_\gamma}(|u|^2u)(s)ds\right\|_{H_x^1} \\
		& \sim \left\|\int_t^\infty e^{i(t-s)\Delta_\gamma}(|u|^2u)(s)ds\right\|_{H_\gamma^1}
		\lesssim \|u\|_{L_t^{q_1}(t,\infty;\dot{W}_x^{\frac{1}{2},r_1})}^2\|u\|_{L_t^{q_2}(t,\infty;W_x^{1,r_2})}
\end{align*}
for any $T \leq t < \infty$, which implies
\begin{align*}
	\lim_{t \rightarrow \infty}\|u(t) - e^{it\Delta_\gamma}u_+\|_{H_x^1}
		= 0.
\end{align*}
In addition, this limit deduces that
\begin{align}
	\lim_{t \rightarrow \infty}\|u(t)\|_{H_\gamma^1}
		= \|u_+\|_{H_\gamma^1}. \label{137}
\end{align}
On the other hand, it follows from Lemma \ref{Time decay} that
\begin{align}
	\|u(t)\|_{L_x^4}
		& \leq \|u(t) - e^{it\Delta_\gamma}u_+\|_{L_x^4} + \|e^{it\Delta_\gamma}u_+\|_{L_x^4} \notag \\
		& \lesssim \|u(t) - e^{it\Delta_\gamma}u_+\|_{H_x^1} + \|e^{it\Delta_\gamma}u_+\|_{L_x^4}
		\longrightarrow 0\ \ \text{ as }\ \ t \rightarrow \infty. \label{138}
\end{align}
Using \eqref{137} and \eqref{138}, we have
\begin{align*}
	\lim_{t \rightarrow \infty}S_{\omega,\gamma}(u(t))
		= \frac{\omega}{2}M[u_+] + K[u_+]
		< S_{\omega,\gamma}(Q_{\omega,\gamma})
\end{align*}
and
\begin{align*}
	\lim_{t \rightarrow \infty}K_{\gamma}(u(t))
		= 2\|\nabla u_+\|_{L^2}^2 + \mu\int_{\mathbb{R}^3}\frac{\gamma}{|x|^\mu}|u_+(x)|^2dx
		> 0.
\end{align*}
Therefore, there exists $T_0 \geq T$ such that
\begin{align*}
	S_{\omega,\gamma}(u(T_0))
		< S_{\omega,\gamma}(Q_{\omega,\gamma})\ \ \text{ and }\ \ 
	K_{\gamma}(u(T_0))
		> 0.
\end{align*}
Applying Theorem \ref{Global well-posedness} and Lemma \ref{Uniform estimate for K}, the solution $u$ to \eqref{NLS} exists globally in time and satisfies
\begin{align*}
	S_{\omega,\gamma}(u(0))
		< S_{\omega,\gamma}(Q_{\omega,\gamma})\ \ \text{ and }\ \ 
	K_{\gamma}(u(0))
		> 0.
\end{align*}
\end{proof}

\section{Linear profile decomposition}\label{Sec:Linear profile decomposition}

In this section, we prove a linear profile decomposition.

\begin{theorem}[Linear profile decomposition]\label{Linear profile decomposition}
Let $d = p = 3$, $\gamma > 0$, and $0 < \mu < 2$.
Let $\{f_n\}$ be a bounded sequence in $H_\text{rad}^1(\mathbb{R}^3)$.
Passing to a subsequence, there exist $J^\ast \in \{0,1,\ldots,\infty\}$, profiles $\{f^j\} \subset H_\text{rad}^1(\mathbb{R}^3)$ satisfying $f^0 \equiv 0$, $f^j \equiv 0$ for any $j \geq 1$ if $J^\ast = 0$, ``$f^j \not\equiv 0$ for any $1 \leq j \leq J^\ast$ and $f^j \equiv 0$ for any $j \geq J^\ast + 1$ if $1 \leq J^\ast < \infty$", and $f^j \not\equiv 0$ for any $1 \leq j < \infty$ if $J^\ast = \infty$, time parameters $\{t_n^j\} \subset \mathbb{R}$, and reminders $\{R_n^J\} \subset H_\text{rad}^1(\mathbb{R}^3)$ such that
\begin{align*}
	f_n
		= \sum_{j=0}^Je^{-it_n^j\Delta_\gamma}f^j + R_n^J
\end{align*}
for any $0 \leq J < \infty$ and any $n \geq 1$.
The time sequence $\{t_n\}$ satisfies
\begin{align}
	\lim_{n \rightarrow \infty}|t_n^j - t_n^k|
		= \infty \label{125}
\end{align}
for any $j \neq k$ and
\begin{align*}
	\text{either }t_n^j \longrightarrow \pm\infty\ \text{ as }\ n \rightarrow \infty
	\ \ \text{ or }\ \ 
	t_n^j \equiv 0\ \text{ for each }\ n \in \mathbb{N}
\end{align*}
for any $j \geq 0$.
Moreover, $\{R_n^J\}$ satisfies
\begin{align*}
	e^{it_n^j\Delta_\gamma}R_n^J
		\xrightharpoonup[]{\hspace{0.4cm}}
		\begin{cases}
		&\hspace{-0.4cm}\displaystyle{
			f^j, \quad (J < j),
		}\\
		&\hspace{-0.4cm}\displaystyle{
			\hspace{0.1cm}0,\hspace{0.1cm} \quad (J \geq j)
		}
		\end{cases}
\end{align*}
in $H_\gamma^1(\mathbb{R}^3)$ as $n \rightarrow \infty$ for any $j \geq 0$ and
\begin{align}
	\lim_{J \rightarrow \infty}\limsup_{n \rightarrow \infty}\|e^{it\Delta_\gamma}R_n^J\|_{L_{t,x}^{q_0}}
		= 0. \label{122}
\end{align}
Furthermore, Pythagorean decomposition
\begin{gather}
	\|(-\Delta_\gamma)^\frac{s}{2}f_n\|_{L^2}^2
		= \sum_{j = 0}^J\|(-\Delta_\gamma)^\frac{s}{2}f^j\|_{L^2}^2 + \|(-\Delta_\gamma)^\frac{s}{2}R_n^J\|_{L^2}^2 + o_n(1)\ \text{ for }\ 0 \leq s \leq 1, \label{123} \\
	\|f_n\|_{L^4}^4
		= \sum_{j = 0}^J\|e^{-it_n^j\Delta_\gamma}f^j\|_{L^4}^4 + \|R_n^J\|_{L^4}^4 + o_n(1) \label{121}
\end{gather}
hold.
\end{theorem}

\begin{proof}
By the interpolation, we have
\begin{align*}
	\|e^{it\Delta_\gamma}f\|_{L_{t,x}^{q_0}}
		\leq \|e^{it\Delta_\gamma}f\|_{L_t^\infty L_x^3}^{1-\theta}\|e^{it\Delta_\gamma}f\|_{L_t^qL_x^r}^\theta
		\lesssim \|e^{it\Delta_\gamma}f\|_{L_t^\infty L_x^3}^{1-\theta}\|f\|_{\dot{H}_x^\frac{1}{2}}^\theta,
\end{align*}
where $(q,r) \in \Lambda_\frac{1}{2}$ and $\theta \in (0,1)$ hold.
Thus, it suffices to show that
\begin{align*}
	\lim_{J \rightarrow \infty}\limsup_{n \rightarrow \infty}\|e^{it\Delta_\gamma}R_n^J\|_{L_t^\infty L_x^3}
		= 0.
\end{align*}
We construct $\{f^j\}$, $\{t_n^j\}$, and $\{R_n^J\}$ inductively.\\
We set $A_0 := \limsup_{n \rightarrow \infty}\|e^{it\Delta_\gamma}f_n\|_{L_t^\infty L_x^3}$.
If $A_0 = 0$, then we get this theorem by taking $f^j \equiv 0$ for each $j \geq 1$.
So, we assume $A_0 > 0$.
Thanks to this assumption, we can take a subsequence of $\{f_n\}$ (, which is denoted by the same symbol) satisfying
\begin{align*}
	\|e^{it\Delta_\gamma}f_n\|_{L_t^\infty L_x^3}
		\geq \frac{3}{4}A_0\ \text{ for each }\ n \in \mathbb{N}.
\end{align*}
Also, we take a sequence $\{t_n^1\} \subset \mathbb{R}$ with
\begin{align}
	\|e^{it_n^1\Delta_\gamma}f_n\|_{L_x^3}
		\geq \frac{1}{2}A_0\ \text{ for each }\ n \in \mathbb{N}. \label{118}
\end{align}
Since $\|e^{it_n^1\Delta_\gamma}f_n\|_{H_\gamma^1}$ is bounded, there exists $f^1 \in H_\gamma^1(\mathbb{R}^3)$ such that $e^{it_n^1\Delta_\gamma}f_n \xrightharpoonup[]{\hspace{0.4cm}}f^1$ in $H_\gamma^1(\mathbb{R}^3)$ as $n \rightarrow \infty$.
From the compactness of $H_{\gamma,\text{rad}}^1(\mathbb{R}^3) \subset L^3(\mathbb{R}^3)$, we get $e^{it_n^1\Delta_\gamma}f_n \longrightarrow f^1$ in $L^3(\mathbb{R}^3)$.
We note that $f^1 \neq 0$ satisfy
\begin{align*}
	\frac{1}{2}A_0
		\leq \|f^1\|_{L^3}
		\lesssim \|f^1\|_{\dot{H}_\gamma^\frac{1}{2}}.
\end{align*}
from \eqref{118}.
We set $R_n^1 := f_n - e^{-it_n^1\Delta_\gamma}f^1$.
Then, we have $R_n^1 \xrightharpoonup[]{\hspace{0.4cm}} 0$ in $H_\gamma^1(\mathbb{R}^3)$ as $n \rightarrow \infty$.
\begin{align*}
	\|R_n^1\|_{\dot{H}_\gamma^s}^2 - \|f_n\|_{\dot{H}_\gamma^s}^2
		= - 2\text{Re}\<f_n,e^{-it_n^1\Delta_\gamma}f^1\>_{\dot{H}_\gamma^s} + \|f^1\|_{\dot{H}_\gamma^s}^2
		\longrightarrow - \|f^1\|_{\dot{H}_\gamma^s}^2\ \text{ as }\ n \rightarrow \infty.
\end{align*}
We set $A_1 := \limsup_{n \rightarrow \infty}\|e^{it\Delta_\gamma}R_n^1\|_{L_t^\infty L_x^3}$.
If $A_1 = 0$, then we get this theorem by taking $f^j \equiv 0$ for each $j \geq 2$.
So, we assume $A_1 > 0$.
Thanks to this assumption, we can take a subsequence of $\{R_n^1\}$ (which is denoted by the same symbol) and a time sequence $\{t_n^2\} \subset \mathbb{R}$ satisfying $\|e^{it_n^2\Delta_\gamma}R_n^1\|_{L_x^3} \geq \frac{1}{2}A_1$ for each $n \geq \mathbb{N}$.
Since $\|e^{it_n^2\Delta_\gamma}R_n^1\|_{H_\gamma^1}$ is bounded, there exists $f^2 \in H_\gamma^1(\mathbb{R}^3)$ such that $e^{it_n^2\Delta_\gamma}R_n^1 \xrightharpoonup[]{\hspace{0.4cm}}f^2$ in $H_\gamma^1(\mathbb{R}^3)$ as $n \rightarrow \infty$.
From the compactness of $H_{\gamma,\text{rad}}^1(\mathbb{R}^3) \subset L^3(\mathbb{R}^3)$, we get $e^{it_n^2\Delta_\gamma}R_n^1 \longrightarrow f^2$ in $L^3(\mathbb{R}^3)$.
We set $R_n^2 := R_n^1 - e^{-it_n^2\Delta_\gamma}f^2$.
Then, we have $R_n^2 \xrightharpoonup[]{\hspace{0.4cm}} 0$ in $H_\gamma^1(\mathbb{R}^3)$ as $n \rightarrow \infty$ and $f_n = \sum_{j = 1}^2e^{-it_n^j\Delta_\gamma}f^j + R_n^2$.
We assume for contradiction that $|t_n^2 - t_n^1|$ is bounded.
Then, there exists a subsequence of $\{t_n^2 - t_n^1\}$ and $t^{2,1} \in \mathbb{R}$ such that $t_n^2 - t_n^1 \longrightarrow t^{2,1}$ as $n \rightarrow \infty$.
Thus, we have
\begin{align*}
	e^{i(t_n^2 - t_n^1)\Delta_\gamma}(e^{it_n^1\Delta_\gamma}f_n - f^1)
		\xrightharpoonup[]{\hspace{0.4cm}}e^{it^{2,1}\Delta_\gamma}0
		= 0\ \text{ as }\ n \rightarrow \infty.
\end{align*}
On the other hand, it follows that
\begin{align*}
	e^{i(t_n^2 - t_n^1)\Delta_\gamma}(e^{it_n^1\Delta_\gamma}f_n - f^1)
		= e^{it_n^2\Delta_\gamma}(f_n - e^{-it_n^1\Delta_\gamma}f^1)
		= e^{it_n^2\Delta_\gamma}R_n^1
		\xrightharpoonup[]{\hspace{0.4cm}}f^2
		\neq 0\ \text{ as }\ n \rightarrow \infty.
\end{align*}
This is contradiction.
Moreover, we have
\begin{align*}
	\|R_n^2\|_{\dot{H}_\gamma^s}^2 - \|f_n\|_{\dot{H}_\gamma^s}^2
		& = \|R_n^1\|_{\dot{H}_\gamma^s}^2 - 2\text{Re}\<R_n^1,e^{-it_n^2\Delta_\gamma}f^2\>_{\dot{H}_\gamma^s} + \|f^2\|_{\dot{H}_\gamma^s}^2 - \|f_n\|_{\dot{H}_\gamma^s}^2 \\
		& \longrightarrow - \|f^1\|_{\dot{H}_\gamma^s}^2 - \|f^2\|_{\dot{H}_\gamma^s}^2\ \text{ as }\ n \rightarrow \infty.
\end{align*}
We construct $\{f^j\}$, $\{t_n^j\}$, and $\{R_n^J\}$ inductively.
We set $A_j := \limsup_{n \rightarrow \infty}\|e^{it\Delta_\gamma}R_n^j\|_{L_t^\infty L_x^3}$.
When there exists $J \geq 0$ such that $A_J = 0$, we get this theorem by taking $f^j \equiv 0$ for each $j \geq J + 1$.
So, we assume $A_j > 0$ for any $j \geq 0$.
We note that
\begin{gather}
	\|e^{it_n^{j+1}\Delta_\gamma}R_n^j\|_{L_x^3}
		\geq \frac{1}{2}A_j\ \text{ for each }\ n \in \mathbb{N}, \notag \\
	e^{it_n^{j+1}\Delta_\gamma}R_n^j \xrightharpoonup[]{\hspace{0.4cm}}f^{j+1}\ \text{ in }\ H_\gamma^1(\mathbb{R}^3)\ \text{ as }\ n \rightarrow \infty, \notag \\
	e^{it_n^{j+1}\Delta_\gamma}R_n^j \longrightarrow f^{j+1}\ \text{ in }\ L^3(\mathbb{R}^3)\ \text{ as }\ n \rightarrow \infty, \notag \\
	\frac{1}{2}A_j
		\leq \|f^{j+1}\|_{L^3}
		\lesssim \|f^{j+1}\|_{\dot{H}_\gamma^\frac{1}{2}}, \label{120} \\
	R_n^{j+1}
		= R_n^j - e^{-it_n^{j+1}}f^{j+1}. \notag
\end{gather}
When $t_n^j$ is bound for some $j \geq 0$, we take a new profile $\~{f}^j := e^{-it^j\Delta_\gamma}f^j$ and a new reminder $\~{R}_n^j := R_n^{j-1} - \~{f}^j$, where $t^j$ satisfies $t_n^j \longrightarrow t^j$ as $n \rightarrow \infty$.
Replacing $e^{-it_n^j\Delta_\gamma}f^j$ with $\~{f}^j$ and $R_n^j$ with $\~{R}_n^j$, we regard as $t_n^j = 0$ for each $n \in \mathbb{N}$.
We assume
\begin{align*}
	\|f_n\|_{\dot{H}_\gamma^s}^2
		= \sum_{j = 0}^J\|f^j\|_{\dot{H}_\gamma^s}^2 + \|R_n^J\|_{\dot{H}_\gamma^s}^2 + o_n(1)
\end{align*}
for induction.
Under this assumption, we have
\begin{align*}
	\|f_n\|_{\dot{H}_\gamma^s}^2
		= \sum_{j = 0}^{J+1}\|f^j\|_{\dot{H}_\gamma^s}^2 + \|R_n^{J+1}\|_{\dot{H}_\gamma^s}^2 + o_n(1).
\end{align*}
Indeed, it follows that
\begin{align*}
	\|R_n^{J+1}\|_{\dot{H}_\gamma^s}^2 - \|f_n\|_{\dot{H}_\gamma^s}^2
		& = \|R_n^J\|_{\dot{H}_\gamma^s}^2 - 2\text{Re}\<R_n^J,e^{-it_n^{J+1}}f^{J+1}\>_{\dot{H}_\gamma^s} + \|f^{J+1}\|_{\dot{H}_\gamma^s}^2 - \|f_n\|_{\dot{H}_\gamma^s}^2 \\
		& \longrightarrow - \sum_{j = 0}^{J+1}\|f^j\|_{\dot{H}_\gamma^s}^2\ \text{ as }\ n \rightarrow \infty.
\end{align*}
If $|t_n^{i+1} - t_n^i|$ is bounded for some $i \geq 1$, then there exists $t^{i+1,i} \in \mathbb{R}$ such that
\begin{align*}
	e^{i(t_n^{i+1} - t_n^i)\Delta_\gamma}(e^{it_n^i\Delta_\gamma}R_n^{i-1} - f^i)
		\xrightharpoonup[]{\hspace{0.4cm}}e^{it^{i+1,i}\Delta_\gamma}0
		= 0\ \text{ as }\ n \rightarrow \infty.
\end{align*}
On the other hand, we have
\begin{align*}
	e^{i(t_n^{i+1} - t_n^i)\Delta_\gamma}(e^{it_n^i\Delta_\gamma}R_n^{i-1} - f^i)
		& = e^{it_n^{i+1}\Delta_\gamma}(R_n^{i-1} - e^{-it_n^i\Delta_\gamma}f^i) \\
		& = e^{it_n^{i+1}\Delta_\gamma}R_n^i
		\xrightharpoonup[]{\hspace{0.4cm}} f^{i+1}
		\neq 0\ \text{ as }\ n \rightarrow \infty.
\end{align*}
Thus, we have $|t_n^{i+1} - t_n^i| \longrightarrow \infty$ as $n \rightarrow \infty$ for any $i \geq 1$.
We assume for induction that $|t_n^k - t_n^i| \longrightarrow \infty$ as $n \rightarrow \infty$ for any $1 \leq i < k$ with $1 \leq |k - i| \leq J$.
Under this assumption, we prove $|t_n^k - t_n^i| \longrightarrow \infty$ as $n \rightarrow \infty$ for any $1 \leq i < k$ with $|k - i| = J + 1$.
If $|t_n^k - t_n^i|$ is bounded for some $1 \leq i < k$ with $|k - i| = J + 1$, then there exists $t^{k,i} \in \mathbb{R}$ such that
\begin{align*}
	e^{i(t_n^k - t_n^i)\Delta_\gamma}(e^{it_n^i\Delta_\gamma}R_n^{i-1} - f^i)
		\xrightharpoonup[]{\hspace{0.4cm}}e^{it^{k,i}\Delta_\gamma}0
		= 0\ \text{ as }\ n \rightarrow \infty.
\end{align*}
On the other hand, we have
\begin{align*}
	e^{i(t_n^k - t_n^i)\Delta_\gamma}(e^{it_n^i\Delta_\gamma}R_n^{i-1} - f^i)
		& = e^{it_n^k\Delta_\gamma}(R_n^{i-1} - e^{-it_n^i\Delta_\gamma}f^i) \\
		& = e^{it_n^k\Delta_\gamma}R_n^i \\
		& = e^{it_n^k\Delta_\gamma}\left(\sum_{j = i+1}^{k-1}e^{-it_n^j\Delta_\gamma}f^j + R_n^{k-1}\right) \\
		& \xrightharpoonup[]{\hspace{0.4cm}}0 + f^k
		= f^k
		\neq 0
\end{align*}
by using Lemma \ref{Time decay}.
This is contradiction.
It follows from \eqref{120} that
\begin{align*}
	\frac{1}{4}\sum_{j=0}^\infty A_j^2
		\lesssim \sum_{j=0}^\infty\|f^j\|_{\dot{H}_\gamma^\frac{1}{2}}^2
		\leq \limsup_{n \rightarrow \infty}\|f_n\|_{\dot{H}_\gamma^\frac{1}{2}}^2,
\end{align*}
so we have $\lim_{j \rightarrow \infty}A_j = 0$.
Finally, we prove \eqref{121}.\\
Case 1. There exists $j_0$ such that $t_n^{j_0} = 0$.
In this case, we show that
\begin{gather*}
	\lim_{n \rightarrow \infty}\|R_n^J\|_{L^4}
		= 0\ \ \text{ for any }\ \ J \geq j_0, \\
	\lim_{n \rightarrow \infty}\|R_n^J\|_{L^4}
		= \|f^{j_0}\|_{L^4}\ \ \text{ for any }\ \ 0 \leq J \leq j_0-1, \\
	\lim_{n \rightarrow \infty}\|e^{-it_n^j\Delta_\gamma}f^j\|_{L^4}
		= 0\ \ \text{ for any }\ \ j \neq j_0, \\
	\lim_{n \rightarrow \infty}\|f_n\|_{L^4}
		= \|f^{j_0}\|_{L^4}.
\end{gather*}
From Lemma \ref{Time decay}, we have the second formula.
We recall $R_n^{j_0-1} \xrightharpoonup[]{\hspace{0.4cm}} f^{j_0}$ in $H_\gamma^1(\mathbb{R}^3)$ as $n \rightarrow \infty$.
By the compactness of $L^4(\mathbb{R}^3) \subset H_\gamma^1(\mathbb{R}^3)$, there exists a subsequence of $\{R_n^{j_0-1}\}$ (, which is denoted by the same symbol) such that $R_n^{j_0-1} \longrightarrow f^{j_0}$ in $L^4(\mathbb{R}^3)$ as $n \rightarrow \infty$.
For $J \geq j_0$, we have
\begin{align*}
	\|R_n^J\|_{L^4}
		& = \Bigl\|R_n^{j_0-1} - \sum_{j=j_0}^Je^{-it_n^j\Delta_\gamma}f^j\Bigr\|_{L^4} \\
		& \leq \|R_n^{j_0-1} - f^{j_0}\|_{L^4} + \sum_{j=j_0+1}^J\|e^{-it_n^j\Delta_\gamma}f^j\|_{L^4}
		\longrightarrow 0\ \ \text{ as }\ \ n \rightarrow \infty.
\end{align*}
For $0 \leq J \leq j_0-1$, we have
\begin{align*}
	\|R_n^J - f^{j_0}\|_{L^4}
		& = \Bigl\|R_n^{j_0} + \sum_{j=J+1}^{j_0-1}e^{-it_n^j\Delta_\gamma}f^j\Bigr\|_{L^4} \\
		& \leq \|R_n^{j_0}\|_{L^4} + \sum_{j=J+1}^{j_0-1}\|e^{-it_n^j\Delta_\gamma}f^j\|_{L^4}
		\longrightarrow 0\ \ \text{ as }\ \ n \rightarrow \infty.
\end{align*}
In addition, we can get
\begin{align*}
	\|f_n - f^{j_0}\|_{L^4}
		& = \Bigl\|\sum_{j=1}^{j_0-1}e^{-it_n^j\Delta_\gamma}f^j - R_n^{j_0}\Bigr\|_{L^4} \\
		& \leq \sum_{j=1}^{j_0-1}\|e^{-it_n^j\Delta_\gamma}f^j\|_{L^4} + \|R_n^{j_0}\|_{L^4}
		\longrightarrow 0\ \ \text{ as }\ \ n \rightarrow \infty.
\end{align*}
Case 2. $|t_n^j| \longrightarrow \infty$ as $n \rightarrow \infty$ holds for any $j$.
In this case, we show that
\begin{gather*}
	\lim_{n \rightarrow \infty}\|e^{-it_n^j\Delta_\gamma}f^j\|_{L^4}
		= 0\ \ \text{ for any }\ \ j, \\
	\lim_{n \rightarrow \infty}\|f_n\|_{L^4}
		= \lim_{n \rightarrow \infty}\|R_n^J\|_{L^4}.
\end{gather*}
The first formula follows from Lemma \ref{Time decay}.
The second formula holds by
\begin{align*}
	\|f_n - R_n^J\|_{L^4}
		= \Bigl\|\sum_{j=1}^Je^{-it_n^j\Delta_\gamma}f^j\Bigr\|_{L^4}
		\leq \sum_{j=1}^J\|e^{-it_n^j\Delta_\gamma}f^j\|_{L^4}
		\longrightarrow 0\ \ \text{ as }\ \ n \rightarrow \infty.
\end{align*}
\end{proof}

\begin{lemma}\label{Estimate for profile decomposition}
Let $d = p = 3$, $\gamma > 0$, $0 < \mu < 2$, $J \in \mathbb{N}$ and let $Q_{\omega,\gamma}$ be the ``radial'' ground state to \eqref{SP}.
If $\{f^j\}_{j=1}^J \subset H_\text{rad}^1(\mathbb{R}^3) \setminus \{0\}$ satisfies
\begin{gather*}
	\sum_{j=1}^J S_{\omega,\gamma}(f^j) - \varepsilon
		\leq S_{\omega,\gamma}\Bigl(\sum_{j=1}^Jf^j\Bigr)
		\leq S_{\omega,\gamma}(Q_{\omega,\gamma}) - \delta, \\
	-\varepsilon
		\leq \mathcal{N}_{\omega,\gamma} \Bigl(\sum_{j=1}^J f^j\Bigr)
		\leq \sum_{j=1}^J \mathcal{N}_{\omega,\gamma}(f^j) + \varepsilon
\end{gather*}
for some $\delta, \varepsilon > 0$ with $2\varepsilon < \delta$, then it follows that
\begin{align*}
	0
		< S_{\omega,\gamma}(f^j)
		< S_{\omega,\gamma}(Q_{\omega,\gamma}),\ \ \ 
	\mathcal{N}_{\omega,\gamma}(f^j)
		> 0
\end{align*}
for each $1 \leq j \leq J$.
\end{lemma}

\begin{proof}
We assume for contradiction that there exists $j \in \{1,\ldots,J\}$ such that $\mathcal{N}_{\omega,\gamma}(f^j)\leq 0$.
From Lemma \ref{Rewriting of infimum}, we have
\begin{align*}
	S_{\omega,\gamma}(Q_{\omega,\gamma})
		& \leq T_{\omega,\gamma}^{1,0}(f^j)
		\leq \sum_{j=1}^J T_{\omega,\gamma}^{1,0}(f^j)
		= \sum_{j=1}^J S_{\omega,\gamma}(f^j) - \frac{1}{2} \sum_{j=1}^J \mathcal{N}_{\omega,\gamma}(f^j)\\
		& \leq S_{\omega,\gamma}(Q_{\omega,\gamma}) - \delta + \varepsilon + \varepsilon
		< S_{\omega,\gamma}(Q_{\omega,\gamma}).
\end{align*}
This is contradiction, so $\mathcal{N}_{\omega,\gamma}(f^j) > 0$ for each $j \in \{1,\ldots,J\}$.
Applying this fact and Lemma \ref{Equivalence of H1 and S}
\begin{align*}
	S_{\omega,\gamma}(f^j)
		> \frac{1}{4}\|f^j\|_{H_{\omega,\gamma}^1}^2
		\geq 0
\end{align*}
and
\begin{align*}
	S_{\omega,\gamma}(f^j)
		\leq \sum_{j=1}^J S_{\omega,\gamma}(f^j)
		\leq S_{\omega,\gamma}(Q_{\omega,\gamma}) - \delta + \varepsilon
		< S_{\omega,\gamma}(Q_{\omega,\gamma}).
\end{align*}
\end{proof}

\section{Scattering}\label{Sec:Scattering}

We define the following sharp scattering threshold level:
\begin{align*}
S_{\omega,\gamma,\text{rad}}^c := \sup\left\{\delta > 0\left|
\begin{array}{l}
\text{If }u_0\in H_\text{rad}^1(\mathbb{R}^3)\text{ satisfies }K_{\gamma}(u_0)\geq 0\text{ and }S_{\omega,\gamma}(u_0)<\delta,\\[0.2cm]
\text{then a solution }u\text{ to \eqref{NLS} scatters.}
\end{array}
\right.\right\}.
\end{align*}
Proposition \ref{Small data scattering} and the fact that the radial solution $e^{i\omega t}Q_{\omega,\gamma}$ to \eqref{NLS} with initial data $Q_{\omega,\gamma}$ does not scatter, imply $S_{\omega,\gamma,\text{rad}}^c\in (0,S_{\omega,\gamma}(Q_{\omega,\gamma})]$.
If we can prove $S_{\omega,\gamma,\text{rad}}^c = S_{\omega,\gamma}(Q_{\omega,\gamma})$, then Theorem \ref{Scattering} holds.
From now on, we assume for contradiction that
\begin{align}
	S_{\omega,\gamma,\text{rad}}^c
		< S_{\omega,\gamma}(Q_{\omega,\gamma}). \label{159}
\end{align}

\begin{lemma}[Palais-Smale condition]\label{Palais-Smale condition}
Let $d = p = 3$, $\gamma > 0$, and $0 < \mu < 2$.
Assume that \eqref{159} holds.
Let $\{u_n : \mathbb{R} \times \mathbb{R}^3 \longrightarrow \mathbb{C}\}$ be a sequence of solution to \eqref{NLS}.
If there exists $\{t_n\} \subset \mathbb{R}$ such that $u_n$ and $t_n$ satisfy
\begin{gather*}
	S_{\omega,\gamma}(Q_{\omega,\gamma})
		> S_{\omega,\gamma}(u_n)
		\searrow S_{\omega,\gamma,\text{rad}}^c,\ \ \ 
	K_\gamma(u_n(t_n))
		> 0, \\
	\lim_{n\rightarrow\infty}\|u_n\|_{L_t^{q_0}(-\infty,t_n;L_x^{q_0})}
		= \lim_{n\rightarrow\infty}\|u_n\|_{L_t^{q_0}(t_n,\infty;L_x^{q_0})}
		= \infty,
\end{gather*}
then passing to a subsequence of $\{u_n(t_n)\}$, $u_n(t_n)$ converges in $H^1(\mathbb{R}^3)$.
\end{lemma}

We note that $u_n$ exists globally in time by Theorem \ref{Global well-posedness}.

\begin{proof}
By the translation in time, we may assume that
\begin{gather*}
	S_{\omega,\gamma}(Q_{\omega,\gamma})
		> S_{\omega,\gamma}(u_{0,n})
		\searrow S_{\omega,\gamma,\text{rad}}^c,\ \ \ 
	K_\gamma(u_{0,n})
		> 0, \\
	\lim_{n\rightarrow \infty}\|u_n\|_{L_t^{q_0}(-\infty,0;L_x^{q_0})}
		= \infty\ \ \text{ and }\ \ 
	\lim_{n\rightarrow \infty}\|u_n\|_{L_t^{q_0}(0,\infty;L_x^{q_0})}
		= \infty,
\end{gather*}
where $u_{0,n}(x) := u_n(0,x)$.
Proposition \ref{Equivalence of H1 and S} implies that $\{u_n\}$ is bounded in $H^1(\mathbb{R}^3)$, so we may apply Theorem \ref{Linear profile decomposition} to $\{u_n\}$, that is, there exist $J^\ast \in \{0, 1, \ldots , \infty\}$, profiles $\{u^j\} \subset H_\text{rad}^1(\mathbb{R}^3)$ satisfying $u^0 \equiv 0$, $u^j \equiv 0$ for any $j \geq 1$ if $J^\ast = 0$, ``$u^j \not\equiv 0$ for any $1 \leq j \leq J^\ast$ and $u^j \equiv 0$ for any $J^\ast+1 \leq j < \infty$ if $1 \leq J^\ast < \infty$'', and $u^j \not\equiv 0$ for any $j \geq 1$ if $J^\ast = \infty$, remainders $\{R_n^J\}$ and parameters $\{t_n^j\}$ such that
\begin{align*}
	u_{0,n}(x)
		= \sum_{j=0}^J e^{-it_n^j\Delta_\gamma}u^j(x) + R_n^J(x)
\end{align*}
for each $0 \leq J \leq \infty$.
If $J^\ast = 0$, then
\begin{align*}
	\lim_{n \rightarrow \infty}\|e^{it\Delta_\gamma}R_n^J\|_{L_{t,x}^{q_0}}
		= \lim_{n \rightarrow \infty}\|e^{it\Delta_\gamma}u_{0,n}\|_{L_{t,x}^{q_0}}
		= 0
\end{align*}
by \eqref{122}.
There exists $n_0 \in \mathbb{N}$ such that
\begin{align*}
	\|e^{it\Delta_\gamma}u_{0,n}\|_{L_{t,x}^{q_0}}
		\leq \varepsilon_0
\end{align*}
for any $n \geq n_0$ and hence,
\begin{align*}
	\|u_n\|_{L_{t,x}^{q_0}}
		\leq 2\varepsilon_0
		< \infty
\end{align*}
by Theorem \ref{Small data global existence}.
However, this is contradiction.
From \eqref{123} and \eqref{121}, we have
\begin{gather}
	S_{\omega,\gamma}(u_{0,n})
		= \sum_{j=1}^JS_{\omega,\gamma}(e^{-it_n^j\Delta_\gamma}u^j) + S_{\omega,\gamma}(R_n^J) + o_n(1), \label{164} \\
	\mathcal{N}_{\omega,\gamma}(u_{0,n})
		= \sum_{j=1}^J\mathcal{N}_{\omega,\gamma}(e^{-it_n^j\Delta_\gamma}u^j) + \mathcal{N}_{\omega,\gamma}(R_n^J) + o_n(1). \notag
\end{gather}
For sufficiently large $n$, there exist $\delta > 0$ and $\varepsilon > 0$ with $2\varepsilon < \delta$ such that
\begin{gather*}
	S_{\omega,\gamma}(u_{0,n})
		\leq S_{\omega,\gamma}(Q_{\omega,\gamma}) - \delta, \\
	S_{\omega,\gamma}(u_{0,n})
		\geq \sum_{j=1}^J S_{\omega,\gamma}(e^{-it_n^j\Delta_\gamma}u^j) + S_{\omega,\gamma}(R_n^J) - \varepsilon, \\
	\mathcal{N}_{\omega,\gamma}(u_{0,n})
		\geq - \varepsilon, \\
	\mathcal{N}_{\omega,\gamma}(u_{0,n})
		\leq \sum_{j=1}^J \mathcal{N}_{\omega,\gamma}(e^{-it_n^j\Delta_\gamma}u^j) + \mathcal{N}_{\omega,\gamma}(R_n^J) + \varepsilon,
\end{gather*}
where we use Proposition \ref{Initial data set} to get the third inequality.
Using Lemma \ref{Estimate for profile decomposition}, we have
\begin{gather*}
	0
		\leq S_{\omega,\gamma}(e^{-it_n^j\Delta_\gamma}u^j)
		< S_{\omega,\gamma,\text{rad}}^c,\ \ \ 
	0
		\leq S_{\omega,\gamma}(R_n^J)
		< S_{\omega,\gamma,\text{rad}}^c, \\
	\mathcal{N}_{\omega,\gamma}(e^{-it_n^j\Delta_\gamma}u^j)
		\geq 0,\ \ \ 
	\mathcal{N}_{\omega,\gamma}(R_n^J)
		\geq 0
\end{gather*}
for sufficiently large $n$.
We note that $S_{\omega,\gamma}(e^{-it_n^j\Delta_\gamma}u^j) = 0$ if and only if $u^j \equiv 0$ by Proposition \ref{Equivalence of H1 and S}.
For $S_{\omega,\gamma}(R_n^J)$, $\mathcal{N}_{\omega,\gamma}(e^{-it_n^j\Delta_\gamma}u^j)$, and $\mathcal{N}_{\omega,\gamma}(R_n^J)$, we can see the same things.
Thus,
\begin{align}
	0
		\leq \lim_{n\rightarrow\infty}S_{\omega,\gamma}(e^{-it_n^j\Delta_\gamma}u^j)
		\leq \lim_{n\rightarrow\infty}S_{\omega,\gamma}(u_{0,n})
		= S_{\omega,\gamma,\text{rad}}
		< S_{\omega,\gamma}(Q_{\omega,\gamma}). \label{124}
\end{align}
When $|t_n^j| \longrightarrow \infty$, we have
\begin{align*}
	\frac{\omega}{2}M[u^j] + \frac{1}{2}K[u^j]
		= S_{\omega,\gamma}(e^{-it_n^j\Delta_\gamma}u^j) + \frac{1}{4}\|e^{-it_n^j\Delta_\gamma}u^j\|_{L^4}
		< S_{\omega,\gamma}(Q_{\omega,\gamma})
\end{align*}
for sufficiently large $n$ by \eqref{124} and Lemma \ref{Time decay}.
From Lemma \ref{Existence of wave operators}, there exists initial data $\~{u}_0^j$ such that
\begin{align*}
	0
		\leq S_{\omega,\gamma}(\~{u}_0^j)
		< S_{\omega,\gamma}(Q_{\omega,\gamma}),\ \ \ 
	K_\gamma(\~{u}_0^j)
		\geq 0,\ \ \ 
	\lim_{n \rightarrow \infty}\|\~{u}^j(- t_n^j) - e^{-it_n^j\Delta_\gamma}u^j\|_{H^1}
		= 0,
\end{align*}
where $\~{u}^j$ is a solution to \eqref{NLS} with initial data $\~{u}_0^j$.
When $t_n^j = 0$, we set $\~{u}_0^j = u^j$.
Then,
\begin{align*}
	0
		\leq S_{\omega,\gamma}(\~{u}_0^j)
		< S_{\omega,\gamma}(Q_{\omega,\gamma}),\ \ \ 
	K_\gamma(\~{u}_0^j)
		\geq 0,\ \ \ 
	\|\~{u}^j(- t_n^j) - e^{-it_n^j\Delta_\gamma}u^j\|_{H^1}
		= 0\ \text{ for any }\ n \in \mathbb{N},
\end{align*}
where $\~{u}^j$ is a solution to \eqref{NLS} with initial data $\~{u}_0^j$.
Then, $\~{u}^j(- t_n^j)$ satisfies
\begin{gather*}
	0
		\leq S_{\omega,\gamma}(\~{u}^j(-t_n^j))
		< S_{\omega,\gamma,\text{rad}}^c, \\
	\mathcal{N}_{\omega,\gamma}(\~{u}^j(-t_n^j))
		\geq 0,\ \ \text{ and }\ \ 
	K_\gamma(\~{u}^j(-t_n^j))
		\geq 0.
\end{gather*}
It follows from the definition of $S_{\omega,\gamma,\text{rad}}^c$ that $\|\~{u}^j\|_{L_{t,x}^{q_0}} < \infty$ for each $j$ and hence, we have $\|\~{u}^j\|_{S_\gamma^1(\mathbb{R})} < \infty$ for each $j$ by Proposition \ref{Persistence of regularity}.
We define a function $\widetilde{u}_n^{\leq J}$ as follows:
\begin{align*}
	\~{u}_n^{\leq J}(t,x)
		:= \sum_{j=1}^J\widetilde{u}_n^j(t,x) + e^{it\Delta_\gamma}R_n^J(x),
\end{align*}
where $\~{u}_n^j(t,x) := \~{u}^j(t-t_n^j,x)$.
We prove that $\~{u}_n^{\leq J}$ is an approximation solution of $u_n$ for sufficiently large $n \in \mathbb{N}$ in the sense of Lemma \ref{Long time perturbation}.
To apply Lemma \ref{Long time perturbation} to $u_n$ and $\~{u}_n^{\leq J}$, we show that the function $\~{u}_n^{\leq J}$ satisfies
\begin{align*}
	\limsup_{n \rightarrow \infty}\{\|\~{u}_n^{\leq J}(0)\|_{H_x^1} + \|\~{u}_n^{\leq J}\|_{L_{t,x}^{q_0}}\}
		\lesssim 1
\end{align*}
uniformly in $J$ and
\begin{align*}
	\lim_{J \rightarrow \infty}\limsup_{n \rightarrow \infty}\||\nabla|^\frac{1}{2}e\|_{L_t^\rho L_x^\kappa}
		= 0,
\end{align*}
where a function $e$ is defined as $e := i\partial_t\~{u}_n^{\leq J} + \Delta_\gamma \~{u}_n^{\leq J} + |\~{u}_n^{\leq J}|^2\~{u}_n^{\leq J}$.
Estimating difference of initial data,
\begin{align*}
	\|\~{u}_n^{\leq J}(0) - u_{0,n}\|_{H^1}
		\leq \sum_{j=1}^J\|\~{u}^j(-t_n^j) - e^{-it_n^j\Delta_\gamma}u^j\|_{H^1}
		\longrightarrow 0
\end{align*}
as $n \rightarrow \infty$ for each $J$.
Since $u_{0,n}$ is bounded in $H^1(\mathbb{R}^3)$, we have
\begin{align*}
	\limsup_{n \rightarrow \infty}\|\~{u}_n^{\leq J}(0)\|_{H^1}
		\lesssim 1
\end{align*}
uniformly in $J$.
The Pythagorean decomposition \eqref{123} deduces
\begin{align*}
	\sum_{j=1}^{\infty}\|u^j\|_{H^1}^2
		\leq \sum_{j=1}^{\infty}\|u^j\|_{H_\gamma^1}^2
		\leq \limsup_{n \rightarrow \infty}\|u_{0,n}\|_{H_\gamma^1}^2
		\lesssim \limsup_{n \rightarrow \infty}\|u_{0,n}\|_{H^1}^2
		\lesssim 1
\end{align*}
and hence, $\lim_{j \rightarrow \infty}\|u^j\|_{H^1} = 0$ holds.
Take $J_0 = J_0(\varepsilon_0)$ such that
\begin{align*}
	\sum_{j=J_0}^{\infty} \|u^j\|_{H^1}^2
		\leq \sum_{j=J_0}^{\infty} \|u^j\|_{H_\gamma^1}^2
		< \varepsilon_0^2,
\end{align*}
where $\varepsilon_0 > 0$ is given in Proposition \ref{Small data global existence}.
Using $\lim_{n \rightarrow \infty}\|\~{u}^j(-t_n^j)\|_{H_\gamma^1} = \|u^j\|_{H_\gamma^1} < \varepsilon_0$ and Proposition \ref{Persistence of regularity}, we have
\begin{align*}
	\limsup_{n \rightarrow \infty}\sum_{j=J_0}^{\infty} \|\~{u}_n^j\|_{S_\gamma^1(\mathbb{R})}^2
		\lesssim \sum_{j=J_0}^{\infty} \|u^j\|_{H^1}^2
		< \varepsilon_0^2.
\end{align*}
Thus, we get
\begin{align}
	\limsup_{n \rightarrow \infty}\sum_{j=1}^{\infty}\|\~{u}_n^j\|_{S_\gamma^1(\mathbb{R})}^2
		\lesssim 1. \label{140}
\end{align}
It follows from the density argument and \eqref{125} that
\begin{align}
	\|\~{u}_n^j\~{u}_n^k\|_{L_{t,x}^\frac{q_0}{2} \cap L_{t,x}^\frac{5}{3}} + \|(-\Delta_\gamma)^\frac{1}{4}\~{u}_n^j(-\Delta_\gamma)^\frac{1}{4}\~{u}_n^k\|_{L_t^\frac{q_1}{2}L_x^\frac{r_1}{2}} + \|(-\Delta_\gamma)^\frac{s}{2}\~{u}_n^j(-\Delta_\gamma)^\frac{s}{2}\~{u}_n^k\|_{L_t^\frac{q_2}{2}L_x^\frac{r_2}{2}}
		\longrightarrow 0 \label{135}
\end{align}
as $n \rightarrow \infty$ for any $j \neq k$ and $0 \leq s \leq 1$.
Using the formula (1.10) in \cite{Ger98} and \eqref{135}, 
\begin{align}
	\Biggl|\Bigl\|\sum_{j=1}^J\~{u}_n^j\Bigr\|_{L_{t,x}^{q_0}}^{q_0} - \sum_{j=1}^J\|\~{u}_n^j\|_{L_{t,x}^{q_0}}^{q_0}\Biggr|
		& = \Biggl|\int_{\mathbb{R}}\int_{\mathbb{R}^3}\Bigl|\sum_{j=1}^J\~{u}_n^j\Bigr|^{q_0} - \sum_{j=1}^J|\~{u}_n^j|^{q_0}dxdt\Biggr| \notag \\
		& \leq \int_{\mathbb{R}}\int_{\mathbb{R}^3}\Biggl|\Bigl|\sum_{j=1}^J\~{u}_n^j\Bigr|^{q_0} - \sum_{j=1}^J|\~{u}_n^j|^{q_0}\Biggr|dxdt \notag \\
		& \lesssim \sum_{1 \leq j \neq k \leq J}\int_{\mathbb{R}}\int_{\mathbb{R}^3}|\~{u}_n^j|^{q_0-1}|\~{u}_n^k|dxdt \notag \\
		& \leq \sum_{1 \leq j \neq k \leq J}\|\~{u}^j\|_{L_{t,x}^{q_0}}^{q_0-2}\|\~{u}_n^j\~{u}_n^k\|_{L_{t,x}^\frac{q_0}{2}} \notag \\
		& \longrightarrow 0\ \text{ as }\ n \rightarrow \infty. \label{139}
\end{align}
From \eqref{139}, Proposition \ref{Small data global existence}, and $\limsup_{n \rightarrow \infty}\|e^{it_n^j\Delta_\gamma}R_n^J\|_{L_{t,x}^{q_0}} \lesssim 1$ uniformly in $J$, we obtain
\begin{align*}
	\limsup_{n \rightarrow \infty}\|\~{u}_n^{\leq J}\|_{L_{t,x}^{q_0}}
		\lesssim 1
\end{align*}
uniformly in $J$.
We prove
\begin{align*}
	\limsup_{n \rightarrow \infty}\|\~{u}_n^{\leq J}\|_{L_{t,x}^\frac{10}{3} \cap L_t^{q_1}\dot{W}_\gamma^{\frac{1}{2},r_1} \cap L_t^{q_2}\dot{W}_\gamma^{s,r_2}}
		\lesssim 1
\end{align*}
uniformly in $J$.
We have
\begin{align*}
	\Bigl\|\sum_{j=1}^J\~{u}_n^j\Bigr\|_{L_t^{q_1}\dot{W}_\gamma^{\frac{1}{2},r_1}}^2
		& = \Bigl\|\sum_{j=1}^J(-\Delta_\gamma)^\frac{1}{4}\~{u}_n^j\Bigr\|_{L_t^{q_1}L_x^{r_1}}^2 \\
		& = \Bigl\|\Bigl(\sum_{j=1}^J(-\Delta_\gamma)^\frac{1}{4}\~{u}_n^j\Bigr)^2\Bigr\|_{L_t^\frac{q_1}{2}L_x^\frac{r_1}{2}} \\
		& \lesssim \sum_{j=1}^J\|\~{u}_n^j\|_{L_t^{q_1}\dot{W}_\gamma^{\frac{1}{2},r_1}}^2 + C_J\sum_{1 \leq j \neq k \leq J}\|(-\Delta_\gamma)^\frac{1}{4}\~{u}_n^j(-\Delta_\gamma)^\frac{1}{4}\~{u}_n^k\|_{L_t^\frac{q_1}{2}L_x^\frac{r_1}{2}}.
\end{align*}
From \eqref{140} and \eqref{135}, $\limsup_{n \rightarrow \infty}\|\~{u}_n^{\leq J}\|_{L_t^{q_1}\dot{W}_\gamma^{\frac{1}{2},r_1}}$ is bounded uniformly in $J$.
The same argument deduces
\begin{align*}
	\limsup_{n \rightarrow \infty}\|\~{u}_n^{\leq J}\|_{L_{t,x}^\frac{10}{3} \cap L_t^{q_2}\dot{W}_\gamma^{s,r_2}}
		\lesssim 1
\end{align*}
uniformly in $J$.
Next, we estimate a error term:
\begin{align*}
	e
		& := i\partial_t\~{u}_n^{\leq J} + \Delta_\gamma \~{u}_n^{\leq J} + |\~{u}_n^{\leq J}|^2\~{u}_n^{\leq J}\\
		& = \Bigl\{F\Bigl(\sum_{j=1}^J\~{u}_n^j + e^{it\Delta_\gamma}R_n^J\Bigr) - F\Bigl(\sum_{j=1}^J\~{u}_n^j\Bigr)\Bigr\} + \Bigl\{F\Bigl(\sum_{j=1}^J\~{u}_n^j\Bigr) - \sum_{j=1}^JF(\~{u}_n^j)\Bigr\}
		=: I_1 + I_2,
\end{align*}
where $F(z) := |z|^2z$.
Since
\begin{align*}
	|\nabla I_2|
		\lesssim \sum_{1 \leq j \neq k \leq J}|\~{u}_n^j|^2|\nabla \~{u}_n^k|,
\end{align*}
we have
\begin{align*}
	\limsup_{n \rightarrow \infty}\|I_2\|_{L_t^2\dot{W}_\gamma^{1,\frac{6}{5}}}
		& \lesssim \limsup_{n \rightarrow \infty}\Bigl\|\sum_{1 \leq j \neq k \leq J}|\~{u}_n^j|^2|\nabla \~{u}_n^k|\Bigr\|_{L_t^2L_x^\frac{6}{5}} \\
		& \lesssim \limsup_{n \rightarrow \infty}\sum_{1 \leq j \neq k \leq J}\|\~{u}_n^j\|_{L_t^{q_1}\dot{W}_x^{\frac{1}{2},r_1}}^2\|\~{u}_n^k\|_{L_t^{q_2}\dot{W}_x^{1,r_2}}
		\lesssim_J 1.
\end{align*}
On the other hand, we have
\begin{align*}
	\limsup_{n \rightarrow \infty}\|I_2\|_{L_{t,x}^\frac{10}{7}}
		& \lesssim \limsup_{n \rightarrow \infty}\Bigl\|\sum_{1 \leq j \neq k \leq J}|\~{u}_n^j|^2|\~{u}_n^k|\Bigr\|_{L_{t,x}^\frac{10}{7}} \\
		& \leq \limsup_{n \rightarrow \infty}\sum_{1 \leq j \neq k \leq J}\|\~{u}_n^j\~{u}_n^k\|_{L_{t,x}^\frac{q_0}{2}}\|\~{u}_n^j\|_{L_{t,x}^\frac{10}{3}}
		= 0
\end{align*}
for each $J$, where we use \eqref{135} to the last equality.
Using the interpolation,
\begin{align*}
	\limsup_{J \rightarrow \infty}\limsup_{n \rightarrow \infty}\|I_2\|_{L_t^\rho \dot{W}_\gamma^{\frac{1}{2},\kappa}}
		\lesssim \limsup_{J \rightarrow \infty}\limsup_{n \rightarrow \infty}\|I_2\|_{L_{t,x}^\frac{10}{7}}^\frac{1}{2}\|I_2\|_{L_t^2\dot{W}_x^{1,\frac{6}{5}}}^\frac{1}{2}
		= 0.
\end{align*}
Next, we estimate $\|I_1\|_{L_t^\rho\dot{W}_\gamma^{\frac{1}{2},\kappa}}$.
\begin{align*}
	\limsup_{n \rightarrow \infty}\|F(\~{u}_n^{\leq J})\|_{L_t^2\dot{W}_\gamma^{1,\frac{6}{5}}}
		\lesssim \limsup_{n \rightarrow \infty}\|\~{u}_n^{\leq J}\|_{L_t^{q_1}\dot{W}_\gamma^{\frac{1}{2},r_1}}^2\|\~{u}_n^{\leq J}\|_{L_t^{q_2}\dot{W}_\gamma^{1,r_2}}
		\lesssim 1
\end{align*}
uniformly in $J$.
Since it follows from \eqref{123} that $\limsup_{n \rightarrow \infty}\|R_n^J\|_{\dot{H}_\gamma^1} \lesssim 1$ uniformly in $J$, we have
\begin{align*}
	& \limsup_{n \rightarrow \infty}\|F(\~{u}_n^{\leq J} - e^{it\Delta_\gamma}R_n^J)\|_{L_t^2\dot{W}_\gamma^{1,\frac{6}{5}}} \\
		& \hspace{2.5cm} \lesssim \limsup_{n \rightarrow \infty}(\|\~{u}_n^{\leq J}\|_{L_t^{q_1}\dot{W}_\gamma^{\frac{1}{2},r_1}} + \|R_n^J\|_{\dot{H}_\gamma^\frac{1}{2}})^2(\|\~{u}_n^{\leq J}\|_{L_t^{q_2}\dot{W}_\gamma^{1,r_2}} + \|R_n^J\|_{\dot{H}_\gamma^1})
		\lesssim 1
\end{align*}
uniformly in $J$.
Combining these inequalities,
\begin{align*}
	\limsup_{n \rightarrow \infty}\|I_1\|_{L_t^2\dot{W}_\gamma^{1,\frac{6}{5}}}
		\lesssim 1
\end{align*}
uniformly in $J$.
On the other hand,
\begin{align*}
	& \limsup_{J \rightarrow \infty}\limsup_{n \rightarrow \infty}\|I_1\|_{L_{t,x}^\frac{10}{7}}
		= \limsup_{J \rightarrow \infty}\limsup_{n \rightarrow \infty}\|F(\~{u}_n^{\leq J}) - F(\~{u}_n^{\leq J} - e^{it\Delta_\gamma}R_n^J)\|_{L_{t,x}^\frac{10}{7}} \\
		& \hspace{1.0cm} \lesssim \limsup_{J \rightarrow \infty}\limsup_{n \rightarrow \infty}\|(|\~{u}_n^{\leq J}|^2 + |e^{it\Delta_\gamma}R_n^J|^2)e^{i\Delta_\gamma}R_n^J\|_{L_{t,x}^\frac{10}{7}} \\
		& \hspace{1.0cm} \leq \limsup_{J \rightarrow \infty}\limsup_{n \rightarrow \infty}\Bigl(\|\~{u}_n^{\leq J}\|_{L_{t,x}^{q_0}}\|e^{it\Delta_\gamma}R_n^J\|_{L_{t,x}^{q_0}}\|\~{u}_n^{\leq J}\|_{L_{t,x}^\frac{10}{3}} + \|e^{it\Delta_\gamma}R_n^J\|_{L_{t,x}^{q_0}}^2\|e^{it\Delta_\gamma}R_n^J\|_{L_{t,x}^\frac{10}{3}}\Bigr) \\
		& \hspace{1.0cm} \leq \limsup_{J \rightarrow \infty}\limsup_{n \rightarrow \infty}\|e^{it\Delta_\gamma}R_n^J\|_{L_{t,x}^{q_0}}\Bigl(\|\~{u}_n^{\leq J}\|_{L_{t,x}^{q_0}}\|\~{u}_n^{\leq J}\|_{L_{t,x}^\frac{10}{3}} + \|e^{it\Delta_\gamma}R_n^J\|_{L_{t,x}^{q_0}}\|e^{it\Delta_\gamma}R_n^J\|_{L_{t,x}^\frac{10}{3}}\Bigr), \\
		& \hspace{1.0cm} \leq \limsup_{J \rightarrow \infty}\limsup_{n \rightarrow \infty}\|e^{it\Delta_\gamma}R_n^J\|_{L_{t,x}^{q_0}}\Bigl(\|\~{u}_n^{\leq J}\|_{L_{t,x}^{q_0}}\|\~{u}_n^{\leq J}\|_{L_{t,x}^\frac{10}{3}} + \|R_n^J\|_{H_\gamma^\frac{1}{2}}\|R_n^J\|_{L^2}\Bigr) \\
		& \hspace{1.0cm} = 0.
\end{align*}
By the interpolation, we have
\begin{align*}
	\limsup_{J\rightarrow \infty}\limsup_{n\rightarrow\infty}\|I_1\|_{L_t^\rho\dot{W}_\gamma^{\frac{1}{2},\kappa}}
		= 0.
\end{align*}
Therefore, we obtain
\begin{align*}
	\lim_{J\rightarrow \infty}\limsup_{n\rightarrow\infty}\||\nabla|^\frac{1}{2}e\|_{L_t^\rho L_x^\kappa}
		= 0.
\end{align*}
Applying Lemma \ref{Long time perturbation} to $u_n$ and $\widetilde{u}_n^{\leq J}$, we get contradiction.
From $J^\ast = 1$,
\begin{align*}
	u_{0,n}(x)
		= e^{-it_n^1\Delta_\gamma}u^1(x) + R_n^1(x).
\end{align*}
We assume $t_n^1 \longrightarrow \infty$ for contradiction.
By \eqref{122} and the monotone convergence, we have
\begin{align*}
	\|e^{it\Delta_\gamma}u_{0,n}\|_{L_t^{q_0}(-\infty,0;L_x^{q_0})}
		\leq \|e^{it\Delta_\gamma}u^1\|_{L_t^{q_0}(-\infty,-t_n^1;L_x^{q_0})} + \|e^{it\Delta_\gamma}R_n^1\|_{L_t^{q_0}(-\infty,0;L_x^{q_0})}
		\longrightarrow 0
\end{align*}
as $n \rightarrow \infty$.
From Proposition \ref{Small data scattering}, we get $\|u_n\|_{L_t^{q_0}(0,\infty;L_x^{q_0})} < \infty$.
This is contradiction.
The case $t_n^1 \longrightarrow -\infty$ is excluded in the same manner.
Thus, we have $t_n^1 \equiv 0$.
Then,
\begin{gather*}
	0
		< S_{\omega,\gamma}(u^1)
		\leq S_{\omega,\gamma,\text{rad}}^c,\ \ \ 
	0
		\leq S_{\omega,\gamma}(R_n^1)
		\leq S_{\omega,\gamma,\text{rad}}^c, \\
	K_{\gamma}(u^1)
		> 0,\ \ \ 
	K_\gamma(R_n^1)
		\geq 0
\end{gather*}
for sufficiently large $n$.
If $S_{\omega,\gamma}(u^1) < S_{\omega,\gamma,\text{rad}}$ holds, then $\|u_n\|_{L_{t,x}^{q_0}} < \infty$ for sufficiently large $n$.
Therefore, we have $S_{\omega,\gamma}(\widetilde{u}_0^1) = S_{\omega,\gamma,\text{rad}}^c$, $K_\gamma(\widetilde{u}_0^1) > 0$, and $\|\widetilde{u}^1\|_{L_t^{q_0}(-\infty,0;L_x^{q_0})} = \|\widetilde{u}^1\|_{L_t^{q_0}(0,\infty;L_x^{q_0})} = \infty$.
We recall that $\widetilde{u}_0^1 = u^1$ and $\widetilde{u}^1$ is a solution to \eqref{NLS} with initial data $\widetilde{u}_0^1$.
From \eqref{164}, $S_{\omega,\gamma}(u^1) = S_{\omega,\gamma,\text{rad}}^c$, $\lim_{n\rightarrow\infty}S_{\omega,\gamma}(u_{0,n}) = S_{\omega,\gamma,\text{rad}}^c$, and Proposition \ref{Equivalence of H1 and S}, we obtain
\begin{align*}
	\lim_{n\rightarrow\infty}\|u_{0,n}-u^1\|_{H^1}
		= \lim_{n\rightarrow \infty}\|R_n^1\|_{H^1}
		\sim \lim_{n\rightarrow\infty}S_{\omega,\gamma}(R_n^1)
		= 0.
\end{align*}
\end{proof}

\begin{proposition}[Existence of a critical solution]\label{Existence of a critical solution}
Let $d = p = 3$, $\gamma > 0$, and $0 < \mu < 2$.
Assume that \eqref{159} holds.
Then, there exists a solution $u_c \in C_t(\mathbb{R} ; H_\text{rad}^1(\mathbb{R}^3))$ to \eqref{NLS} such that
\begin{align*}
	S_{\omega,\gamma}(u_c)
		= S_{\omega,\gamma,\text{rad}}^c,\ \ \ 
	K_\gamma(u_c)
		> 0,\ \ \ 
	\|u_c\|_{L_t^{q_0}(-\infty,0;L_x^{q_0})}
		= \|u_c\|_{L_t^{q_0}(0,\infty;L_x^{q_0})}
		= \infty,
\end{align*}
and
\begin{align*}
	K
		= \{u(t)\in H_\text{rad}^1(\mathbb{R}^3):t \in \mathbb{R}\}
\end{align*}
is precompact in $H^1(\mathbb{R}^3)$.
\end{proposition}

\begin{proof}
From \eqref{159}, there exists a sequence of initial data $\{u_{0,n}\} \subset H_\text{rad}^1(\mathbb{R}^3)$ such that
\begin{align*}
	S_{\omega,\gamma}(Q_{\omega,\gamma})
		> S_{\omega,\gamma}(u_{0,n})
		\searrow S_{\omega,\gamma,\text{rad}},\ \ \ 
	K_\gamma(u_{0,n})
		> 0,\ \ \ 
	\|u_n\|_{L_t^{q_0}L_x^{q_0}}
		= \infty
\end{align*}
for each $n \in \mathbb{N}$, where $u_n$ denotes a solution to \eqref{NLS} with initial data $u_{0,n}$.
We note that $u_n$ exists globally in time by Theorem \ref{Global well-posedness}.
There exists $\{\tau_n\} \subset \mathbb{R}$ such that
\begin{align*}
	\lim_{n\rightarrow \infty}\|u_n\|_{L_t^{q_0}(-\infty,\tau_n;L_x^{q_0})}
		= \infty\ \ \text{ and }\ \ 
	\lim_{n\rightarrow \infty}\|u_n\|_{L_t^{q_0}(\tau_n,\infty;L_x^{q_0})}
		= \infty.
\end{align*}
Proposition \ref{Equivalence of H1 and S} implies that $\{u_n\}$ is bounded in $H^1(\mathbb{R}^3)$, so we may apply Theorem \ref{Linear profile decomposition} to $\{u_n\}$, that is, there exist $J^\ast \in \{0,1,\ldots,\infty\}$, profiles $\{u^j\} \subset H_\text{rad}^1(\mathbb{R}^3)$, remainders $\{R_n^J\}$, and parameter $\{t_n^j\}$ such that
\begin{align*}
	u_{0,n}
		= \sum_{j=0}^J e^{-it_n^j\Delta_\gamma}u^j + R_n^J
\end{align*}
for each $0 \leq J \leq \infty$.
By the proof of Lemma \ref{Palais-Smale condition}, we have $J^\ast = 1$.
If we set $u_{c,0} = \~{u}_0^1$ and $u_c = \~{u}^1$, then $u_{c,0}$ and $u_c$ satisfy
\begin{align*}
	S_{\omega,\gamma}(u_c)
		= S_{\omega,\gamma,\text{rad}}^c,\ \ \ 
	K_\gamma(u_c)
		> 0,\ \ \ 
	\|u_c\|_{L_t^{q_0}(-\infty,0;L_x^{q_0})}
		= \|u_c\|_{L_t^{q_0}(0,\infty;L_x^{q_0})}
		= \infty,
\end{align*}
where $\~{u}_0^1$ and $\~{u}^1$ are constructed in the proof of Lemma \ref{Palais-Smale condition}.
We take any time sequence $\{t_n\} \subset \mathbb{R}$.
Then, $\{u_c(t_n)\} \subset H_\text{rad}^1(\mathbb{R}^3)$ satisfies
\begin{align*}
	S_{\omega,\gamma}(u_c(t_n))
		= S_{\omega,\gamma,\text{rad}}^c,\ \ \ 
	K_\gamma(u_c(t_n))
		> 0,\ \ \ 
	\|u_c\|_{L_t^{q_0}(-\infty,t_n;L_x^{q_0})}
		= \|u_c\|_{L_t^{q_0}(t_n,\infty;L_x^{q_0})}
		= \infty.
\end{align*}
Therefore, passing to a subsequence of $\{u_c(t_n)\}$ if necessary, $u_c(t_n)$ converges in $H^1(\mathbb{R}^3)$.
\end{proof}

\begin{lemma}\label{Lower boundedness of H1}
Let $d = p = 3$, $\gamma > 0$, and $0 < \mu < 2$.
Let $u_0 \in H_\text{rad}^1(\mathbb{R}^3)$.
If a solution $u$ to \eqref{NLS} with \eqref{IC} satisfies that
\begin{align*}
	K
		= \{u(t) \in H_\text{rad}^1(\mathbb{R}^3) : t \in \mathbb{R}\}
\end{align*}
is precompact in $H^1(\mathbb{R}^3)$, then there exists $c > 0$ such that
\begin{align*}
	\|\nabla u(t)\|_{L_x^2}
		\geq c\,M[u_0]
\end{align*}
for any $t \in \mathbb{R}$.
\end{lemma}

\begin{proof}
When $u_0 \equiv 0$, this lemma holds.
We assume $u_0 \neq 0$.
If this lemma does not hold in this case, then there exists $\{t_n\} \subset \mathbb{R}$ such that
\begin{align*}
	\|\nabla u(t_n)\|_{L^2}
		< \frac{1}{n}M[u_0].
\end{align*}
Thus, we have $u(t_n) \longrightarrow 0$ in $\dot{H}^1(\mathbb{R}^3)$ as $n \rightarrow \infty$.
On the other hand, there exists a subsequence of $\{u(t_n)\}$ such that $u(t_n)$ converges in $H^1(\mathbb{R}^3)$ since the set $K$ is precompact in $H^1(\mathbb{R}^3)$.
Combining these facts, $u(t_n) \longrightarrow 0$ in $H^1(\mathbb{R}^3)$ as $n \rightarrow \infty$ and hence, we have
\begin{align*}
	0
		= \lim_{n \rightarrow \infty}\|u(t_n)\|_{L^2}
		= \lim_{n \rightarrow \infty}\|u_0\|_{L^2}
		= \|u_0\|_{L^2}.
\end{align*}
However, this is contradiction.
\end{proof}

\begin{proposition}\label{Concentration on the origin}
Let $d = p = 3$, $\gamma > 0$, and $0 < \mu < 2$.
If a solution $u$ to \eqref{NLS} satisfies that
\begin{align*}
	K
		= \{u(t) \in H_\text{rad}^1(\mathbb{R}^3) : t \in \mathbb{R}\}
\end{align*}
is precompact in $H^1(\mathbb{R}^3)$, then for any $\varepsilon > 0$, there exists $R > 0$ such that
\begin{align*}
	\|u(t)\|_{L^2(|x| \geq R)} + \|\nabla u(t)\|_{L^2(|x| \geq R)} + \|u(t)\|_{L^4(|x| \geq R)}
		\leq \varepsilon
\end{align*}
for any $t \in \mathbb{R}$.
\end{proposition}

\begin{proof}
If not, then there exists $\varepsilon > 0$ such that for any $n \in \mathbb{N}$, there exists $t_n \in \mathbb{R}$ such that
\begin{align*}
	\|u(t_n)\|_{L^2(|x| \geq n)} + \|\nabla u(t_n)\|_{L^2(|x| \geq n)} + \|u(t_n)\|_{L^4(|x| \geq n)}
		> \varepsilon.
\end{align*}
Since the set $K$ is precompact in $H^1(\mathbb{R}^3)$, there exist a subsequence of $\{t_n\} \subset \mathbb{R}$ and $u_\infty \in H^1(\mathbb{R}^3)$ such that $u(t_n) \longrightarrow u_\infty$ in $H^1(\mathbb{R}^3)$ as $n \rightarrow \infty$.
From $u_\infty \in H^1(\mathbb{R}^3)$, we have
\begin{align*}
	\|u_\infty\|_{L^2(|x| \geq n)} + \|\nabla u_\infty\|_{L^2(|x| \geq n)} + \|u_\infty\|_{L^4(|x| \geq n)}
		< \frac{1}{2}\varepsilon
\end{align*}
for sufficiently large $n \in \mathbb{N}$.
By $u(t_n) \longrightarrow u_\infty$ in $H^1(\mathbb{R}^3)$ as $n \rightarrow \infty$, it follows that
\begin{align*}
	\|u(t_n) - u_\infty\|_{L^2} + \|\nabla u(t_n) - \nabla u_\infty\|_{L^2} + \|u(t_n) - u_\infty\|_{L^4}
		< \frac{1}{2}\varepsilon
\end{align*}
for sufficiently large $n \in \mathbb{N}$.
Therefore, we obtain
\begin{align*}
	\|u(t_n)\|_{L^2(|x| \geq n)} + \|\nabla u(t_n)\|_{L^2(|x| \geq n)} + \|u(t_n)\|_{L^4(|x| \geq n)}
		< \frac{1}{2}\varepsilon + \frac{1}{2}\varepsilon
		= \varepsilon.
\end{align*}
This is contradiction.
\end{proof}

\begin{theorem}[Rigidity]\label{Rigidity}
Let $d = p = 3$, $\gamma > 0$, $0 < \mu < 2$, and $u_0 \in H_\text{rad}^1(\mathbb{R}^3)$.
If there exists $\omega > 0$ such that
\begin{align*}
	S_{\omega,\gamma}(u_0)
		< S_{\omega,\gamma}(Q_{\omega,\gamma}),\ \ 
	K_\gamma(u_0)
		\geq 0,
\end{align*}
and
\begin{align*}
	K
		= \{u(t) \in H_\text{rad}^1(\mathbb{R}^3) : t \in \mathbb{R}\}
\end{align*}
is precompact in $H^1(\mathbb{R}^3)$, then $u_0 \equiv 0$.
\end{theorem}

\begin{proof}
We assume $u_0 \not\equiv 0$ for contradiction.
Then, $K_\gamma(u_0) > 0$ holds by the definition of $r_{\omega,\gamma}$ and $S_{\omega,\gamma}(u_0) < r_{\omega,\gamma} = S_{\omega,\gamma}(Q_{\omega,\gamma})$.
From Lemmas \ref{Estimate of K from below} and \ref{Lower boundedness of H1}, we have
\begin{align}
	K_\gamma(u(t))
		& \geq \min\left\{S_{\omega,\gamma}(Q_{\omega,\gamma}) - S_{\omega,\gamma}(u_0), \frac{2\mu}{7}\|(-\Delta_\gamma)^\frac{1}{2}u(t)\|_{L^2}^2\right\} \notag \\
		& \geq \min\left\{S_{\omega,\gamma}(Q_{\omega,\gamma}) - S_{\omega,\gamma}(u_0), c\,\|\nabla u(t)\|_{L_x^2}^2\right\} \notag \\
		& \geq \min\left\{S_{\omega,\gamma}(Q_{\omega,\gamma}) - S_{\omega,\gamma}(u_0), c\,M[u_0]\right\}
		=: \delta_0
		> 0. \label{126}
\end{align}
We note that $\delta_0$ is independent of $t$.
We define the following functions for $R > 0$.
A cut-off function $\mathscr{X}_R \in C_c^\infty(\mathbb{R}^3)$ is radially symmetric and satisfies
\begin{equation*}
\mathscr{X}_R(r) := R^2\mathscr{X}\left(\frac{r}{R}\right),\ \text{ where }\ \mathscr{X}(r) :=
\begin{cases}
\hspace{-0.4cm}&\displaystyle{
		\hspace{0.45cm}r^2\hspace{0.45cm} \quad (0 \leq r \leq 1),
	} \\
\hspace{-0.4cm}&\displaystyle{
		smooth \quad (1 \leq r \leq 3),
	} \\
\hspace{-0.4cm}&\displaystyle{
		\hspace{0.53cm}0\hspace{0.53cm} \quad (3 \leq r),
	}
\end{cases}
\end{equation*}
$\mathscr{X}''(r) \leq 2$ $(r \geq 0)$, and $r=|x|$.
For $\mathscr{X}_R(x)$, we define a functional
\begin{align*}
	I(t)
		:= \int_{\mathbb{R}^3}\mathscr{X}_R(x)|u(t,x)|^2dx.
\end{align*}
Then, it follows from Proposition \ref{Virial identity} that
\begin{align*}
	I'(t)
		= 2\text{Im}\int_{\mathbb{R}^3}\frac{R}{r}\mathscr{X}'\left(\frac{r}{R}\right)\overline{u(t,x)}x\cdot\nabla u(t,x)dx
\end{align*}
and
\begin{align}
	I''(t)
		= 4K_\gamma(u(t)) + \mathcal{R}_1 + \mathcal{R}_2 + \mathcal{R}_3 + \mathcal{R}_4, \label{127}
\end{align}
where $\mathcal{R}_k = \mathcal{R}_k(t)$ $(k = 1, 2, 3, 4)$ are defined as
\begin{align*}
	\mathcal{R}_1
		& := 4\int_{\mathbb{R}^3}\left\{\mathscr{X}''\left(\frac{r}{R}\right) - 2\right\}|\nabla u(t,x)|^2dx, \\
	\mathcal{R}_2
		& := - \int_{\mathbb{R}^3}\left\{\mathscr{X}''\left(\frac{r}{R}\right) + \frac{2R}{r}\mathscr{X}'\left(\frac{r}{R}\right) - 6 \right\}|u(t,x)|^4dx, \\
	\mathcal{R}_3
		& := - \int_{\mathbb{R}^3}\left\{\frac{1}{R^2}\mathscr{X}^{(4)}\left(\frac{r}{R}\right) + \frac{4}{Rr}\mathscr{X}^{(3)}\left(\frac{r}{R}\right)\right\}|u(t,x)|^2dx, \\
	\mathcal{R}_4
		& := 2\mu\int_{\mathbb{R}^3}\left\{\frac{R}{r}\mathscr{X}'\left(\frac{r}{R}\right) - 2\right\}\frac{\gamma}{|x|^\mu}|u(t,x)|^2dx.
\end{align*}
Applying Proposition \ref{Concentration on the origin}, there exists $R > 1$ such that
\begin{align}
	|R_1 + R_2 + R_3 + R_4|
		\lesssim \int_{|x| \geq R}\left(|\nabla u(t,x)|^2 + |u(t,x)|^4 + \frac{1}{R^\mu}|u(t,x)|^2\right)dx
		< \frac{1}{2}\delta_0 \label{128}
\end{align}
for any $t \in \mathbb{R}$.
From \eqref{126}, \eqref{127}, and \eqref{128}, there exists $R > 0$ such that
\begin{align}
	I''(t)
		\geq \frac{1}{2}\delta_0
		> 0 \label{129}
\end{align}
for any $t \in \mathbb{R}$.
We fix such $R > 0$.
\eqref{129} implies $\lim_{t \rightarrow \infty}I'(t) = \infty$.
On the other hand, we have
\begin{align*}
	|I'(t)|
		\lesssim R\,\|u(t)\|_{L_x^2}\|\nabla u(t)\|_{L_x^2}
		\leq R\,\|u(t)\|_{H_x^1}^2
		\sim R\,S_{\omega,\gamma}(u(t))
		\lesssim R
\end{align*}
by Proposition \ref{Equivalence of H1 and S}.
Therefore, we obtain contradiction.
\end{proof}

\begin{proof}[Proof of Theorem \ref{Scattering}]
Let $u_c$ be the critical solution given in Theorem \ref{Existence of a critical solution}.
From Theorem \ref{Existence of a critical solution}, $u_c$ satisfies the all assumptions in Theorem \ref{Rigidity}.
Therefore, $u_c \equiv 0$ holds.
However, this is contradiction with $\|u_c\|_{L_{t,x}^{q_0}} = \infty$.
\end{proof}

\subsection*{Acknowledgements}
M.H. is supported by JSPS KAKENHI Grant Number JP19J13300.
M.I. is supported by JSPS KAKENHI Grant Number JP18H01132, JP19K14581, and JST CREST Grant Number JPMJCR1913.

\end{document}